\documentclass[11pt,reqno]{amsart}
\usepackage[a4paper,scale=.75,twoside=false]{geometry}  
\usepackage[dvipsnames]{xcolor}                                     

\usepackage{amsfonts}                                   
\usepackage{amsmath}                                    
\usepackage{amssymb}                                    
\usepackage{amsthm}                                     
\usepackage[british]{babel}                             
\usepackage[style=ext-numeric-comp,giveninits=true,articlein=false]{biblatex}
\usepackage{caption}                                    
\usepackage{csquotes}
\usepackage[shortlabels,inline]{enumitem}               
\usepackage[T1]{fontenc}                                
\usepackage{graphicx}                                   
\usepackage{lmodern}                                    
\usepackage{mathtools}                                  
\usepackage{microtype}                                  
\usepackage{siunitx}                                    

\usepackage{pgf}
\usepackage{tikzexternal}

\tikzsetexternalprefix{tikz/}
\newcommand{\map}[3]{ #1 \colon #2 \to #3 }

\newcommand{\mb}[1]{\mathbb{#1}}
\newcommand{\mc}[1]{\mathcal{#1}}

\newcommand{\mr}[1]{\mathrm{#1}}
\newcommand{\R}{ \mb{R} }

\newcommand{\N}{ \mb{N} }

\DeclarePairedDelimiter{\parn}{\lparen}{\rparen}        
\DeclarePairedDelimiter{\brac}{\lbrace}{\rbrace}        
\DeclarePairedDelimiter{\brak}{\lbrack}{\rbrack}        

\DeclarePairedDelimiter{\abs}{\lvert}{\rvert}           
\DeclarePairedDelimiter{\norm}{\lVert}{\rVert}          

\DeclarePairedDelimiter{\floor}{\lfloor}{\rfloor}       

\DeclareMathOperator{\arsinh}{arsinh}
\DeclareMathOperator*{\elimsup}{ess\,lim\,sup}
\DeclareMathOperator{\pv}{p.v.}
\DeclareMathOperator{\re}{Re}
\DeclareMathOperator{\sgn}{sgn}

\newcommand{\dd}{\mathop{}\!\mr{d}}                     

\newcommand{\ol}[1]{\overline{#1}}
\newcommand{\ul}[1]{\underline{#1}}

\newcommand{\ceq}{\coloneqq}
\newcommand{\eqc}{\eqqcolon}
              
\usepackage[colorlinks=true,urlcolor=NavyBlue,citecolor=ForestGreen,linkcolor=BrickRed]{hyperref} 
\usepackage[nameinlink,capitalise]{cleveref}                           

\theoremstyle{plain}
\newtheorem{theorem}{Theorem}[section]
\newtheorem*{main-theorem}{Main Theorem}

\newtheorem{corollary}[theorem]{Corollary}
\newtheorem{lemma}[theorem]{Lemma}
\newtheorem{proposition}[theorem]{Proposition}

\theoremstyle{definition}

\theoremstyle{remark}
\newtheorem{remark}[theorem]{Remark}
\newtheorem{assumption}{Assumption}

\crefname{assumption}{assumption}{assumptions}      
\crefname{enumi}{part}{parts}                       

\numberwithin{equation}{section}                    

\title[On the behaviour of extreme solutions]{On the precise cusped behaviour of extreme solutions to Whitham-type equations}

\author[M. Ehrnström]{Mats Ehrnström}
\email{mats.ehrnstrom@ntnu.no}
\address{Department of Mathematical Sciences, Norwegian University of Science and Technology, 7491 Trondheim, Norway}

\author[O. Mæhlen]{Ola I. H. Mæhlen}
\email{ola.mahlen@ntnu.no}
\address{Department of Mathematical Sciences, Norwegian University of Science and Technology, 7491 Trondheim, Norway}

\author[K. Varholm]{Kristoffer Varholm}
\email{kristoffer.varholm@ntnu.no}
\address{Department of Mathematical Sciences, Norwegian University of Science and Technology, 7491 Trondheim, Norway}

\thanks{The authors acknowledge the support by grant nos. 250070, 301538 and 325114 from the Research Council of Norway. Part of this research was carried out during the programme \emph{Mathematical Problems in Fluid Dynamics} at MSRI, Berkeley. The authors are thankful to the referees for suggestions which helped improve the presentation of the paper.}


\begin{document}
\begin{abstract}
    We prove exact leading-order asymptotic behaviour at the origin for nontrivial solutions of two families of nonlocal equations. The equations investigated include those satisfied by the cusped highest steady waves for both the uni- and bidirectional Whitham equations. The problem is therefore analogous to that of capturing the \ang{120} interior angle at the crests of classical Stokes' waves of greatest height. In particular, our results partially settle conjectures for such extreme waves posed in a series of recent papers \cite{Ehrnstroem19Existence,Ehrnstroem19Whithams,Truong22Global}. Our methods may be generalised to solutions of other nonlocal equations, and can moreover be used to determine asymptotic behaviour of their derivatives to any order.
\end{abstract}
\maketitle

\section{Introduction}
    The \emph{Whitham equation}
    \begin{equation}
        \label{eq:unidirectionalWhitham}
        \partial_t \phi + \partial_x\parn*{K_W * \phi + \phi^2} = 0,
    \end{equation}
    where $\phi$ represents the surface profile and
    \[
        \hat{K}_W(\xi) = \int_{\R} K_W(x)e^{-ix\xi}\dd x \ceq \sqrt{\frac{\tanh(\xi)}{\xi}},
    \]
    is a fully dispersive variant of the classical Korteweg--de Vries (KdV) equation, originally proposed in \cite{Whitham67Variational}. It features some properties that the KdV equation lacks, such as wave breaking \cite{Hur17Wave,Saut22Wave}, highest waves \cite{Ehrnstroem19Whithams,Truong22Global,Ehrnstroem23Direct}, and better high-frequency modelling \cite{Emerald21Rigorousa}. While Whitham added the dispersion in an ad hoc manner, the model has since been both justified experimentally and derived from the full water-wave problem in several ways. See for instance \cite{Klein18Whitham,Moldabayev15Whitham,Carter18Bidirectional}, in addition to the aforementioned \cite{Emerald21Rigorousa}.

    Another water-wave model is similarly obtained by making the Boussinesq system -- of which the KdV equation can be viewed as a unidirectional version -- fully dispersive, so as to arrive at the Whitham--Boussinesq system
    \begin{equation}
        \label{eq:bidirectionalWhitham}
        \begin{aligned}
            \partial_t \phi + \partial_x\parn*{K_B * v + \phi v} & = 0, \\
            \partial_t v + \partial_x\parn*{\phi + v^2/2}        & = 0,
        \end{aligned}
    \end{equation}
    also called the \emph{bidirectional Whitham equation}. Here, $\phi$ again denotes the surface profile, $v$ relates to the fluid velocity at the surface, and the convolution kernel $K_B$ is defined by its symbol
    \begin{equation}
        \label{eq:bidirectionalWhithamSymbol}
        \hat{K}_B \ceq \hat{K}_W^2 = \frac{\tanh(\xi)}{\xi}.
    \end{equation}
    Strictly speaking, there are several ways to make the Boussinesq system fully dispersive, but \eqref{eq:bidirectionalWhitham} represents one of the natural candidates that have been investigated in the literature, see e.g. \cite{AcevesSanchez13Numerical, Pei19Note, Nilsson19Solitary,Emerald21Rigorous}. It is also currently the only of these fully dispersive systems that is known to admit highest steady waves \cite{Ehrnstroem19Existence}.

    Various steady solutions to the Whitham equations, both the uni- and bidirectional, have been found and studied. Of particular interest to us here are the global, locally analytic curves of periodic steady waves found in \cite{Ehrnstroem19Existence,Ehrnstroem19Whithams}, bifurcating from the line of trivial waves and approaching a so-called limiting \emph{highest} wave. These are waves whose height reach the maximal value of $c/2$ for the unidirectional Whitham equation, and $c^2/3$ for the bidirectional Whitham equation, where $c$ denotes the velocity of the wave. In the full water-wave problem, it is part of the famous Stokes' conjecture that the analogous highest Stokes' waves have angled crests, with interior angles of exactly $120^\circ$. That is, a highest steady wave with a crest at the origin satisfies
    \[
        \varphi(0)-\varphi(x) = \parn*{\frac{1}{\sqrt{3}}+o(1)}\abs{x}
    \]
    as $x \to 0$. This was ultimately proved in \cite{Plotnikov02Proof,Amick82Stokes}.

    For the Whitham equation \eqref{eq:unidirectionalWhitham}, it was conjectured by Whitham\footnote{With a minor error in the exact constant, which was pointed out in \cite{Ehrnstroem19Whithams}.} \cite{Whitham74Linear} that the local behaviour of an analogous highest wave should instead be the cusped variant
    \begin{equation}
        \label{eq:WhithamBehavior}
        \frac{c}{2}-\varphi(x) = \parn*{\sqrt{\frac{\pi}{8}}+o(1)}\abs{x}^{1/2}
    \end{equation}
    as $x \to 0$. The authors of \cite{Ehrnstroem15Whithams, Ehrnstroem19Whithams} were able to determine that there indeed was a highest periodic wave $\varphi$ for the Whitham equation. Furthermore, they showed that any bounded solution reaching that height must satisfy both
    \begin{equation}
        \label{eq:liminfAndLimsupOfHighestWaveForWhithamEquation}
        0<\liminf_{x \to 0}{\frac{c/2-\varphi(x)}{\abs{x}^{1/2}}} \quad \text{and} \quad \limsup_{x \to 0}{\frac{c/2-\varphi(x)}{\abs{x}^{1/2}}} < \infty,
    \end{equation}
    but did not establish the full limit described in \eqref{eq:WhithamBehavior}.

    More recently, the existence of full global curves of \textit{solitary} waves up to a highest wave has also been proved \cite{Truong22Global,Ehrnstroem23Direct}. The same asymptotic estimates \eqref{eq:liminfAndLimsupOfHighestWaveForWhithamEquation} from \cite{Ehrnstroem19Whithams} apply equally well for these. Furthermore, there is also an innovative computer-assisted proof \cite{Enciso18Convexity}, where a highest periodic wave satisfying the limiting behaviour \eqref{eq:WhithamBehavior} is constructed. A form of local uniqueness, and the convexity of this highest wave is also obtained. The idea is to build an approximate ansatz for the solution using special functions, sufficiently good for a fixed-point argument to go through. A very large number of terms is required, as the map involved is just barely a contraction. A recent paper in the same direction for the Burgers--Hilbert equation is \cite{Dahne23Highest}.

    As for what concerns the bidirectional Whitham equation \eqref{eq:bidirectionalWhitham}, it was shown in \cite{Ehrnstroem19Existence} that there exists a highest periodic wave $\varphi$ with a corresponding $v$ that satisfies
    \[
        \limsup_{x \to 0}{\frac{(1-1/\sqrt{3})c-v(x)}{\abs{x}\log(1/\abs{x})}} < \infty.
    \]
    The corresponding lower bound is stated, but a flaw in one of the preceding lemmas hinders a correct estimate. This is due to slightly subtle estimates where logarithmic factors are easily lost, making the proof more delicate than for the unidirectional Whitham equation.

    The main purpose of this paper is to provide an analytic, and relatively transparent, argument establishing both the limit \eqref{eq:WhithamBehavior} for the Whitham equation, and the analogous result
    \begin{equation}
        \label{eq:bidirectionalWhithamBehavior}
        \frac{c^2}{3}-\varphi(x) = \parn*{\frac{1}{3\pi}+o(1)}\abs{x}\log(1/\abs{x})
    \end{equation}
    as $x \to 0$, for the bidirectional Whitham equation. These results will follow from a somewhat more general method for calculating the local behaviour of solutions to two classes of nonlocal equations on the half-line. As the proofs are quite technical, we first provide some background; describing how these nonlinear waves are related to the more general formulation and results found in \Cref{sec:setup,sec:logarithmicKernel,sec:homogeneousKernel}.

\section{Background and overview}\label{sec:overview}
    The Whitham equation \eqref{eq:unidirectionalWhitham} is a prototypical example of a more general family of nonlocal, nonlinear shallow-water wave models of form
    \begin{equation}
        \label{eq:generalequation}
        \partial_t \phi + \partial_x \parn*{K * \phi + N(\phi)} = 0,
    \end{equation}
    where $K \in L^1(\R)$ is an even, positive integral kernel that is convex on $\R^+ \ceq (0,\infty)$. Generally, this kernel will arise from a Fourier multiplier symbol $\hat{K}(\xi)$ of negative order. We shall here consider the orders $-1$ and $-1/2$, appearing in the bi- and unidirectional gravity water wave problems, respectively \cite{Lannes13Water}. The decay and smoothness of the symbol is realised as a corresponding singularity at the origin of an otherwise smooth kernel $K$ of exponential decay; see \cite{Taylor11Partial,Grafakos14Modern}.

    Seeking steady solutions $\phi(t,x) = \varphi(x-ct)$ to \eqref{eq:generalequation}, one arrives at
    \begin{equation}
        \label{eq:generalEquationSteady}
        K * \varphi = f(\varphi) + A, \quad \text{where} \quad f(t) \ceq ct - N(t),
    \end{equation}
    for some constant $A \in \R$ after integration. Under quite general conditions, equations like \eqref{eq:generalequation} have only symmetric solitary waves of elevation \cite{Bruell17Symmetry,Arnesen22Decay}. A similar statement is true for steady periodic waves under an additional reflection assumption \cite{Bruell23Symmetry}. This property is inherited via the maximum principle for the elliptic convolution operator. Naturally, this leads to the study of a possible maximal height $\varphi(0)$ for solutions of \eqref{eq:generalEquationSteady}.

    Supposing now that $f$ is increasing to the left of a nondegenerate local maximum at $t = \gamma$, and is sufficiently smooth, we can write
    \[
        f(\gamma)-f(t) = \parn*{-\frac{1}{2}f''(\gamma) + g(\gamma-t)}(\gamma-t)^2
    \]
    with $g(0)=0$. Thus, if $\varphi$ is a solution to \eqref{eq:generalEquationSteady} that achieves $\varphi(0) = \gamma$ from below, then $u = \gamma - \varphi$ is a nonnegative solution to
    \[
        K * u - (K * u)(0) = \parn*{-\frac{1}{2}f''(\gamma) + g(u(x))}u(x)^2,
    \]
    vanishing at the origin.

    Motivated by this computation, we therefore consider the condensed equation
    \begin{equation}
        \label{eq:condensedFormulation}
        (1+n(u(x)))u(x)^2 = \int_\R (K(y-x) - K(y))u(y)\dd y,
    \end{equation}
    where $K$ again has the properties as described after \eqref{eq:generalequation}, and $n(0)=0$. We see that any pointwise solution will necessarily have to satisfy $u(0)=0$. Equations similar to the one in \eqref{eq:condensedFormulation} also appear in a plethora of other contexts: examples include harmonic, functional and stochastic analysis.

    Finally, note further that \eqref{eq:condensedFormulation} is equivalent to the equation
    \begin{equation}
        \label{eq:condensedFormulationEven}
        (1+n(u(x)))u(x)^2 = \int_0^\infty \delta_x^2 K(y)u(y)\dd y
    \end{equation}
    for even functions, where we have conveniently recognised the second-order central difference
    \begin{equation}
        \label{eq:secondDifference}
        \delta_x^2 K(y) \ceq K(y+x)+K(y-x)-2K(y)
    \end{equation}
    in the integrand. Whereas the first-order difference in \eqref{eq:condensedFormulation} is very useful when one wants to establish global estimates for $u$, \eqref{eq:condensedFormulationEven} is able to take direct advantage of the convexity of $K$. It is therefore especially well adapted for studying $u$ precisely at $x = 0$.

    We will consider \eqref{eq:condensedFormulationEven} under general assumptions, but first \emph{formally} outline the theory below for kernels capturing the same singular behaviour. The exact assumptions and rigorous statements follow in \Cref{sec:setup}.

    \subsection*{Homogeneous singularity (Whitham)}
        If we replace $K$ in \eqref{eq:condensedFormulationEven} with the homogeneous, but merely locally integrable
        \begin{equation}
            \label{eq:homogeneousSingularity}
            H_s(x) \ceq \abs{x}^{s-1}
        \end{equation}
        for $s \in (0,1)$, and let $n = 0$, we obtain the toy equation
        \begin{equation}
            \label{eq:homogeneousToyEquation}
            u(x)^2 = \int_0^\infty \delta_x^2 H_s(y) u(y)\dd y,
        \end{equation}
        which in fact has an \emph{explicit} unbounded solution. It is convenient to introduce
        \[
            \Phi_s(\tau) \ceq \delta_1^2 H_s(\tau)=\abs{\tau+1}^{s-1}+\abs{\tau-1}^{s-1}-2\abs{\tau}^{s-1},
        \]
        for then the second difference in \eqref{eq:homogeneousToyEquation} satisfies
        \begin{equation}
            \label{eq:secondDifferenceHomogeneity}
            \delta_x^2 H_s(\tau x) = H_s(x)\Phi_s(\tau).
        \end{equation}

        \begin{lemma}
            \label{lem:toyEquation}
            The toy equation \eqref{eq:homogeneousToyEquation} has
            \[
                u(x) = \beta_s\abs{x}^s, \qquad \beta_s \ceq \frac{1}{2}B(s,s)
            \]
            as a solution for every $s \in (0,1)$, where $B$ denotes the beta function.
        \end{lemma}
        \begin{proof}
            This is an immediate consequence of the identity
            \begin{equation}
                \label{eq:betaIntegralIdentity}
                \beta_s = \int_0^\infty \Phi_s(\tau)\tau^s\dd \tau,
            \end{equation}
            which can most easily be seen for $s \in (0,1/2)$ by splitting the integral according to
            \begin{multline*}
                \int_0^\infty \Phi_s(\tau)\tau^s\dd \tau = \underbrace{\int_0^1 (1-\tau)^{s-1}\tau^s\dd\tau}_{\beta_s} - \underbrace{\int_0^1 \tau^{2s-1}\dd \tau}_{1/(2s)}\\+ \underbrace{\int_0^\infty \parn[\big]{(1+\tau)^{s-1}-\tau^{s-1}}\tau^s\dd\tau}_{I_1} + \underbrace{\int_1^\infty \parn[\big]{(\tau-1)^{s-1}-\tau^{s-1}}\tau^s\dd\tau,}_{I_2}
            \end{multline*}
            where all but the first term will cancel.

            Indeed, observe that
            \[
                I_2 = \frac{1}{s} - \int_0^\infty \parn[\big]{\tau^s - (\tau + 1)^s}(\tau+1)^{s-1}\dd\tau
            \]
            through the change of variables $\tau \mapsto (\tau + 1)$ and integration by parts. It follows that
            \[
                I_1 + I_2 = \frac{1}{s} + \int_0^\infty \parn[\big]{(\tau+1)^{2s-1}-\tau^{2s-1}}\dd\tau = \frac{1}{2s},
            \]
            whence \eqref{eq:betaIntegralIdentity} holds. Finally, analytic continuation yields \eqref{eq:betaIntegralIdentity} also for $s \in [1/2,1)$.
        \end{proof}

        In particular, it is reasonable to expect that well-behaved solutions to \eqref{eq:condensedFormulationEven} should still satisfy
        \begin{equation}
            \label{eq:homogeneousCaseLimit}
            \lim_{x \to 0} \frac{u(x)}{\abs{x}^{1/2}} = \beta_{1/2} =  \frac{\pi}{2}
        \end{equation}
        when $K$ behaves like $H_{1/2}$ near the origin; which is the case for a scaled version of the Whitham-kernel $K_W$. Under mild conditions, equations such as \eqref{eq:condensedFormulation} have the feature that solutions are smooth away from where they vanish. This comes from a general ``off-diagonal'' convolution property for pseudo-differential operators \cite{Taylor11Partial}, and can be seen as in \cite{Ehrnstroem19Whithams}.

        The behaviour of a solution in the vicinity of the origin arises from a balancing act between the square on the left-hand side, and the asymptotics of the second difference \eqref{eq:secondDifference} as $x \to 0$. As the square root is not regular, one consequently faces an upper threshold on the regularity of $u$. Simplifying to \eqref{eq:homogeneousToyEquation}, an essential part of the argument in \cite{Ehrnstroem19Whithams} relies on first bootstrapping global $C^{1/2-}$-regularity, and then noting that
        \begin{align}
            \parn*{\frac{u(x)}{\abs{x}^\alpha}}^2 & = \abs{x}^{1/2-\alpha}\int_0^\infty \Phi(\tau)\tau^\alpha \frac{u(\tau x)}{\abs{\tau x}^\alpha}\dd\tau \label{eq:toyEquationSubstitution} \\
                                                  & \leq \abs{x}^{1/2-\alpha} \sup_{y \in \R}{\frac{u(y)}{\abs{y}^{\alpha}}} \int_0^\infty \abs{\Phi(\tau)}\tau^\alpha \dd \tau \notag
        \end{align}
        for all $\alpha \in (0,1/2)$ and $x \neq 0$, where
        \begin{equation}
            \label{eq:homogeneousPhiDefinition}
            \Phi(\tau) \ceq \Phi_{1/2}(\tau) = \frac{1}{\abs{1+\tau}^{1/2}} + \frac{1}{\abs{1-\tau}^{1/2}}-\frac{2}{\abs{\tau}^{1/2}}.
        \end{equation}
        If we now, for the sake of argument, \emph{assume} that the supremum of the left-hand side in \eqref{eq:toyEquationSubstitution} is always achieved for $\abs{x} \leq 1$, then we obtain
        \begin{equation}
            \label{eq:formalUpperBound}
            \sup_{x\in \R}{\frac{u(x)}{\abs{x}^{\alpha}}} \leq \int_0^\infty \abs{\Phi(\tau)}\tau^\alpha \dd \tau,
        \end{equation}
        whereupon we can let $\alpha \to 1/2$.

        A curious thing about this calculation is that if $\Phi$ had been non-negative, then \eqref{eq:formalUpperBound} would have immediately yielded
        \[
            u(x) \leq \frac{\pi}{2}\abs{x}^{1/2}
        \]
        by \eqref{eq:betaIntegralIdentity}, which would be optimal. Similarly, if one knew that the limit of $u(x)/\abs{x}^{1/2}$ existed as $x \to 0$, one could have chosen $\alpha = 1/2$ in \eqref{eq:toyEquationSubstitution} and let $x \to 0$ to find \eqref{eq:homogeneousCaseLimit} by dominated convergence. In reality, however, $\Phi$ changes from negative to positive at a point $\tau_0 \in (0,1)$, as seen in \Cref{fig:PhiGraph}; and the existence of a limit is exactly what is difficult to show.

        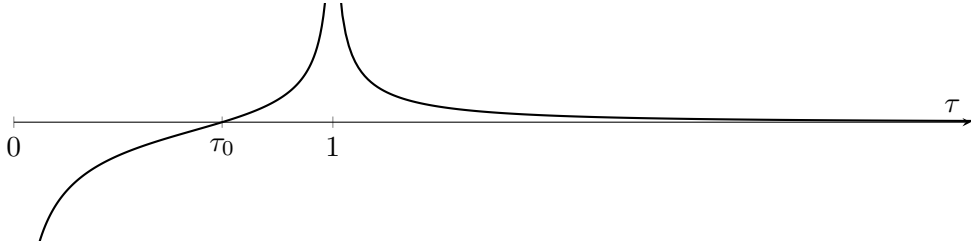
\begin{figure}[htb]
            \centering
            \tikzsetnextfilename{homogeneous_phi}
            \begin{tikzpicture}[trim axis left, trim axis right]
    \begin{axis}[xlabel={$\tau$}, ylabel ={$\Phi(\tau)$},xmin=0,xmax=3,xtick={0,0.6529,1},xticklabels={0,$\tau_0$,1},ymin=-5,ymax=5,ytick=\empty,width=0.9\textwidth,height=0.3\textwidth,axis x line = middle,axis y line=none]
        \addplot[mark=none,thick,black,domain=0.01:0.99,samples=200] {(1+x)^(-1/2)+(1-x)^(-1/2)-2*x^(-1/2)};
        \addplot[mark=none,thick,black,domain=1.01:5,samples=400] {(1+x)^(-1/2)+(x-1)^(-1/2)-2*x^(-1/2)};
    \end{axis}
\end{tikzpicture}
            \caption{The graph of $\Phi$.}
            \label{fig:PhiGraph}
        \end{figure}

        To establish \eqref{eq:homogeneousCaseLimit}, we therefore identify in \eqref{eq:betaIntegralIdentity} the significance of the points $\tau \in \brac{\tau_0,1}$, and write \eqref{eq:condensedFormulationEven} as
        \begin{align}
            \parn*{\frac{u(x)}{x^{1/2}}}^2(1+n(u(x))) & = \frac{1}{\abs{x}} \parn*{\int_0^{\tau_0 x} + \int_{\tau_0 x}^x + \int_x^\nu + \int_\nu^\infty} \delta_x^2 K(y)u(y)\dd y \notag
            \intertext{for $0 < x < \nu$, or, in essence,}
            \parn*{\frac{u(x)}{x^{1/2}}}^2            & \approx \parn*{\int_0^{\tau_0} + \int_{\tau_0}^1 + \int_1^{\nu/x} + \int_{\nu/x}^\infty} \Phi(\tau)\tau^{1/2}\frac{u(\tau x)}{(\tau x)^{1/2}} \dd \tau \label{eq:keyApproximateEquality}
        \end{align}
        under appropriate assumptions. The constant $\nu > 0$ is used to single out a small interval where $u$ has desirable properties, but can otherwise be made arbitrarily small. Its exact value is therefore not important to the theory. Because $\nu / x \to \infty$ as $x \searrow 0$, the last integral will vanish in the limit.

        The remaining integrals are less straightforward, and the main obstacle in their treatment is the limited information about monotonicity or the existence of the limit. Our trick here is to consider sequences realising
        \[
            m \ceq \liminf_{x \searrow 0}{\frac{u(x)}{x^{1/2}}} \quad \text{or} \quad M \ceq \limsup_{x \searrow 0}{\frac{u(x)}{x^{1/2}}},
        \]
        in a strategic manner. As $\Phi$ changes signs at $\tau_0$, we are thereby able to make the estimates
        \begin{equation}
            \label{eq:symmetric_inequalities}
            \begin{aligned}
                M^2 & \leq m\int_0^{\tau_0} \Phi(\tau)\tau^{1/2}\dd \tau + M \int_{\tau_0}^\infty \Phi(\tau)\tau^{1/2}\dd \tau,     \\
                m^2 & \geq  M \int_0^{\tau_0} \Phi(\tau) \tau^{1/2} \dd \tau + m\int_{\tau_0}^\infty \Phi(\tau) \tau^{1/2}\dd \tau,
            \end{aligned}
        \end{equation}
        by taking limits in \eqref{eq:keyApproximateEquality}.

        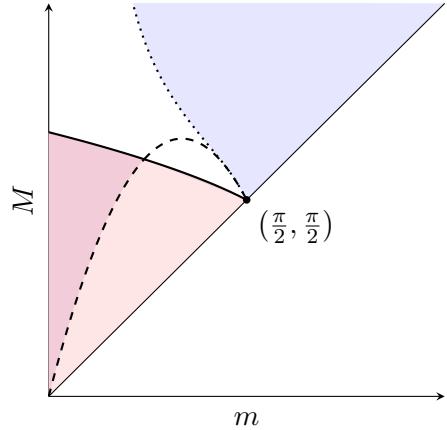
\begin{figure}[htb]
            \centering
            \tikzsetnextfilename{inequalities}
            \begin{tikzpicture}[trim axis left, trim axis right]
    \begin{axis}[xlabel={$m$}, ylabel ={$M$},xmin=0,xmax=pi,xtick=\empty,ymin=0,ymax=pi,ytick=\empty,width=0.5\textwidth,axis equal image,axis x line=bottom,axis y line=left]
        \pgfmathsetmacro{\firstint}{0.541146}
        \pgfmathsetmacro{\secondint}{2.11194}
        \pgfmathsetmacro{\mintersect}{0.757324}
        \addplot[name path=upperbound,mark=none,domain=0:pi/2,thick,black] {\secondint/2+(\secondint^2/4-x*\firstint)^(1/2)};
        \addplot[name path=lowerbound_right,mark=none,domain=\mintersect:pi/2,thick,black,dashed] {1/\firstint*x*(\secondint-x)};
        \addplot[name path=lowerbound_left,mark=none,domain=0:\mintersect,thick,black,dashed] {1/\firstint*x*(\secondint-x)};
        \addplot[name path=lowerbound_improved,mark=none,thick,black,dotted] table[col sep=comma, row sep=newline] {graphics/lowerbound.txt};
        \addplot[name path=xleft,mark=none, domain=0:pi/2] {x};
        \addplot[name path=xright,mark=none, domain=pi/2:pi] {x};
        \fill[black] (axis cs:pi/2,pi/2) circle (0.05 cm) node[below right] {$\parn*{\frac{\pi}{2},\frac{\pi}{2}}$};
        \addplot[fill = red!10] fill between [of=upperbound and xleft];
        \addplot[fill=purple!20] fill between[of=lowerbound_left and upperbound];
        \addplot[fill = blue!10] fill between [of=lowerbound_improved and xright,reverse=true];
    \end{axis}
\end{tikzpicture}
            \caption{The inequalities in \eqref{eq:symmetric_inequalities} are satisfied by the points below the solid curve, and above the dashed curve, respectively. The refined version of the second inequality corresponds to the dotted curve.}
            \label{fig:inequalities}
        \end{figure}

        This system of inequalities will have solutions described by \Cref{fig:inequalities}. In addition to the expected solution, which is isolated, there is also a wedge-like set of unwanted solutions for which $M > m$. A refinement is made to the second inequality of \eqref{eq:symmetric_inequalities} to exclude this area, yielding the desired conclusion that $m=M=\pi/2$. The shape of the curves in \Cref{fig:inequalities} is naturally determined by integrals involving $\Phi_s$, but there is some leeway. Therefore, this method works essentially unmodified for a range of homogeneous singularities; not only when $s=1/2$. In fact, there is some $s_0 \approx 1/3$ such that it works for $s \in (s_0,1)$, but fails for $s \in (0,s_0)$. The reason why it breaks down is that the expected solution $m=M=\beta_s$ from \Cref{lem:toyEquation} stops being isolated, even with the refined inequality. A new idea would therefore be required to proceed past this value. Highest Hölder and Lipschitz waves have been constructed in a number of settings \cite{Afram21Steady,Oerke22Highest,Arnesen19Non,Le22Waves,Bruell21Waves,Hildrum23Periodic,Geyer19Linear}, and we expect a similar approach to go through for many such equations.

    \subsection*{Logarithmic singularity (Bidirectional Whitham)}
        For order $-1$, the singular behaviour of the kernel is instead captured by
        \begin{equation}
            \label{eq:logarithmicSingularity}
            L(x) \ceq \log\parn*{1/\abs{x}},
        \end{equation}
        so that homogeneity is replaced by additivity. One finds that
        \begin{equation}
            \label{eq:secondDifferenceAdditivity}
            \delta_x^2 L(\tau x) = -\log{\abs*{1-\frac{1}{\tau^2}}} \eqc \Lambda(\tau)
        \end{equation}
        differs substantially from \eqref{eq:secondDifferenceHomogeneity}, in that does not depend on $x$ at all. This leads to an entirely different set of estimates, and, in turn, changes the relative importance of the integrals appearing in governing equation.

        The qualitative behaviour of $\Lambda$ is still the same as $\Phi$ in \Cref{fig:PhiGraph}, but in the logarithmic case the contribution of the entire interval $(0,x)$ turns out to be negligible in the limit. In fact, the final estimate hinges only on an integral over $(x,\nu)$. Explicitly, writing \eqref{eq:condensedFormulationEven} as
        \[
            \parn*{\frac{u(x)}{x\log(1/x)}}^2(1+n(u(x))) = \frac{1}{x^2 \log(1/x)^2}\parn*{\int_0^x + \int_x^\nu + \int_\nu^\infty} \delta_x^2 K(y) u(y)\dd y,
        \]
        we see that the analogue of \eqref{eq:keyApproximateEquality} becomes
        \[
            \parn*{\frac{u(x)}{x\log(1/x)}}^2 \approx \parn*{\int_0^1 + \int_1^{\nu/x} + \int_{\nu/x}^\infty}\Lambda(\tau)\tau\frac{\log(1/(\tau x))}{\log(1/x)^2}\frac{u(\tau x)}{\tau x \log(1/(\tau x))} \dd \tau
        \]
        for $0 < x < \nu$.

        The first integral is killed in the limit when $u(x)/(x\log(1/x))$ is bounded, and the third integral is still negligible as before. The limit ``should'' therefore be
        \begin{equation}
            \label{eq:logCaseLimit}
            \lim_{x \to 0}{\frac{u(x)}{\abs{x}\log(1/\abs{x})}} = \lim_{x \searrow 0}{\int_1^{\nu / x} \Lambda(\tau) \tau \frac{\log(1/(\tau x))}{\log(1/x)^2} \dd \tau} = \frac{1}{2}
        \end{equation}
        in this case. Contrary to what we saw for a homogeneous singularity, it is possible to obtain \eqref{eq:logCaseLimit} directly using the aforementioned approach with sequences. There is no need to thereafter go through a system of inequalities.

\section{Setup}\label{sec:setup}
    We have seen that, after an appropriate change of variables, the highest waves of both \eqref{eq:unidirectionalWhitham} and \eqref{eq:bidirectionalWhitham} satisfy an equation of the form \eqref{eq:condensedFormulationEven}. The following assumptions are made on the objects involved, where $\R_0^+ \ceq \R^+ \cup \brac{0}$:
    \begin{assumption}
        \label{ass:nonlinearity}
        The nonlinearity $n\in C^1(\R_0^+)$ satisfies $n(0)=0$.
    \end{assumption}

    \begin{assumption}
        \label{ass:integrabilitySymmetryAndConvexity}
        The kernel $K \in L^1(\R)$ is even, positive, and convex on $\R^+$. Moreover, it admits a decomposition $K=S+R$, where the singular part $S$ is of the form
        \[
            L(x) = \log(1/\abs{x}) \quad \text{or} \quad H(x) = \abs{x}^{-1/2},
        \]
        and the regular part $R$ has a weak second derivative $R''\in L^1(\R)$.
    \end{assumption}
    \begin{assumption}
        \label{ass:solutionU}
        The solution $u\in C(\R_0^+)$ is bounded, nonnegative, satisfies $u(0)=0$, and is for sufficiently small $x>0$ both continuously differentiable and increasing.
    \end{assumption}
    The regularity of $n$ in \Cref{ass:nonlinearity} is only needed for the limit of the derivative in our main result, not for the behaviour of $u(x)$ itself. One similarly would need to demand higher regularity of $n$ in order to prove asymptotics for higher order derivatives of $u$. The assumption of continuous differentiability of $u$ close to the origin in \Cref{ass:solutionU} is in fact redundant under the other properties, see \Cref{sec:homogeneousKernel}. All of the assumptions may be further weakened, but as added generality would come at the expense of clarity, we shall not push this question further here.

    Our main result is the following.
    \begin{main-theorem}
    \phantomsection\label{thm:mainTheorem} 
    Suppose that the above \Cref{ass:nonlinearity,ass:integrabilitySymmetryAndConvexity,ass:solutionU}  hold. If the singular part of $K$ takes the logarithmic form $L(x)=\log(1/\abs{x})$, then the solution $u$ admits the limits
    \begin{alignat*}{3}
        \lim_{x\to 0}\frac{u(x)}{x\log (1/x)} & = \frac{1}{2} \quad   &  & \text{and} & \quad\lim_{x\to 0}\frac{u'(x)}{\log (1/x)} & =\frac{1}{2},
        \intertext{while if it takes the homogeneous form $H(x)=|x|^{-1/2}$, then $u$ admits the limits}
        \lim_{x\to 0}\frac{u(x)}{x^{1/2}}     & = \frac{\pi}{2} \quad &  & \text{and} & \quad \lim_{x\to 0} \frac{u'(x)}{x^{-1/2}} & =\frac{\pi}{4}.
    \end{alignat*}
    \end{main-theorem}

    This theorem combines \Cref{prop:logarithmicAsymptoticBehavior,prop:logarithmicDerivativeAsymptoticBehavior,prop:homogeneousAsymptoticBehavior,prop:homogeneousderivative}, of which the first two are proved in \Cref{sec:logarithmicKernel} and the latter two in \Cref{sec:homogeneousKernel}. These results are in turn employed in \Cref{sec:applicationForTheHighestWavesOfBidirectionalWhithamEquation} to both establish the limit \eqref{eq:bidirectionalWhithamBehavior} and global regularity for the bidirectional highest waves obtained in \cite{Ehrnstroem19Existence}, and in \Cref{sec:applicationForHighestWavesOfUnidirectionalWhithamEquations} to prove the limit \eqref{eq:WhithamBehavior} for the unidirectional highest waves obtained in \cite{Ehrnstroem19Whithams,Truong22Global}. Furthermore, immediately preceding \Cref{sec:applicationForHighestWavesOfUnidirectionalWhithamEquations}, we outline how one would determine the asymptotic behaviour of derivatives to any order.

    For reference, we include the following corollary, which lists the implied asymptotic behaviour for the highest waves in the Whitham equations. The precise details are, as explained above, presented in \Cref{sec:applicationForTheHighestWavesOfBidirectionalWhithamEquation,sec:applicationForHighestWavesOfUnidirectionalWhithamEquations}.

    \begin{corollary}[Whitham equations, abridged]
        Let $\varphi$ denote the surface profile of a highest wave, with a peak at zero, of the Whitham equation. Then
        \begin{align*}
            \lim_{x\to 0}\frac{\varphi(0)-\varphi(x)}{|x|^{1/2}} = \sqrt{\frac{\pi}{8}}.
        \end{align*}
        The corresponding limit for a highest wave in the bidirectional Whitham equation is
        \begin{align*}
            \lim_{x\to 0}\frac{\varphi(0)-\varphi(x)}{|x|\log(1/|x|)} = \frac{1}{3\pi}.
        \end{align*}
    \end{corollary}

    \subsection{Preliminaries}
        We will here list a few useful properties of the kernel $K=S+R$ that follow from \Cref{ass:integrabilitySymmetryAndConvexity}. The first lemma shows that the tail of $\delta_x^2K$ is both nonnegative and small, which will later ensure that it may be disregarded when analysing the local behaviour of $u$ near the origin. Introducing the antiderivative
        \begin{equation}
            \label{eq:K_antiderivative}
            \mc{K}(x) \ceq \int_0^x K(y)\dd y
        \end{equation}
        for the kernel will occasionally be useful.

        \begin{lemma}
            \label{lem:secondDifferenceProperties}
            The second difference $\delta_x^2 K$ is nonnegative on $(x,\infty)$, and satisfies
            \[
                0\leq \int_\nu^\infty \delta_x^2 K(y)\dd y \leq -K'(\nu-x)x^2,
            \]
            for any $0\leq x<\nu$.
        \end{lemma}
        \begin{proof}
            For $0 \leq x < y$, the convexity of $K$ on $\R^+$ immediately yields
            \[
                \delta_x^2K(y) = K(y+x)+K(y-x)- 2K(y) \geq 0,
            \]
            and by virtue of \eqref{eq:K_antiderivative}, we get
            \begin{align*}
                \int_\nu^\infty \delta_x^2 K(y) \dd y & = - \delta_x^2 \mc{K}(\nu)= - \int_0^x \int_0^x K'(\nu + \tau_1 - \tau_2) \\
                                                      & \leq - K'(\nu-x)x^2
            \end{align*}
            for all $0 \leq x < \nu$. Here we have used that $-K'$ is nonincreasing on $\R^+$.
        \end{proof}

        The following lemma likewise demonstrates that $\delta_x^2R$ is small, and so it too will have a negligible effect on the local behaviour of $u$.

        \begin{lemma}
            \label{lem:remainderMayBeDisregarded}
            The second difference $\delta_x^2 R$ is integrable and satisfies
            \[
                \norm{\delta_x^2 R}_{L^1} \leq x^2 \norm{R''}_{L^1}
            \]
            for all $x \in \R_0^+$. Moreover, $R'$ admits the bound $\norm{R'}_{L^\infty}\leq \norm{R''}_{L^1}$.
        \end{lemma}
        \begin{proof}
            By \Cref{ass:integrabilitySymmetryAndConvexity}, $R''$ is integrable, and so
            \begin{align*}
                \norm{\delta_x^2 R}_{L^1}
                 & = \int_\R\abs[\bigg]{\int_{0}^x\int_0^xR''(y+t_1-t_2)\dd t_1\dd t_2}\dd y                      \\
                 & \leq \int_{0}^x\int_0^x\int_{\R}\abs{R''(y+t_1-t_2)}\dd y\dd t_1\dd t_2= x^2 \norm{R''}_{L^1}.
            \end{align*}
            for all $x \geq 0$. For the second part, we note that $R=K-S$ is necessarily even; and so $R'$ is odd. Since $R'$ is also absolutely continuous, we therefore conclude that $R'(x)= \int_0^xR''(y)\dd y$, which gives the desired bound.
        \end{proof}

\section{Logarithmic kernel}
    \label{sec:logarithmicKernel}
    In this section, we adopt \Cref{ass:integrabilitySymmetryAndConvexity,ass:nonlinearity,ass:solutionU}, and specifically assume that the singular part of the kernel $K$ is of the form $S(x)=L(x) =\log(1/\abs{x})$. Additionally, we restrict $x$ to an interval $(0,\nu]$ throughout, for some $0 < \nu \ll 1$ such that $u$ is continuously differentiable and increasing on $(0,\nu]$. This is possible due to \Cref{ass:solutionU}.

    Since we will prove the first two limits of \hyperref[thm:mainTheorem]{Main Theorem} here, we naturally introduce the shorthand
    \begin{equation}
        \label{eq:definitionOfEll}
        \ell(x) \ceq x\log(1/x),
    \end{equation}
    and
    \begin{equation}
        \label{eq:logarithmicgDefinition}
        g(x) \ceq \frac{u(x)}{\ell(x)},
    \end{equation}
    which is well-defined for all $x\in(0,\nu]$. We also adopt the function $\Lambda$ from \eqref{eq:secondDifferenceAdditivity}, whose utility comes from the identity
    \begin{equation}
        \label{eq:identityForKernelLogarithmicCase}
        \delta_x^2K(\tau x)= \delta_x^2L(\tau x)+\delta_x^2R(\tau x)=\Lambda(\tau) + \delta_x^2R(\tau x),
    \end{equation}
    which holds by linearity of $\delta_x^2$.

    Seeking to determine the limit of $g$ at zero, we begin with a lemma that asymptotically rephrases \eqref{eq:condensedFormulationEven} in terms of $g$.
    \begin{lemma}
        \label{lem:theLemmaThatProvidesTheAsymptoticEquationForLogarithmicG}
        With $\ell$ and $g$ as in \eqref{eq:definitionOfEll} and \eqref{eq:logarithmicgDefinition}, respectively, we have the equation
        \begin{equation}
            \label{eq:asymptoticEquationForLogarithmicG}
                (1+o(1))g(x)^2=o(1)g(x)+\int_x^{\nu} \brak[\Bigg]{\frac{\delta_x^2 K(y)\ell(y)}{\ell(x)^2}}g(y)\dd y
                +\int_\nu^\infty\frac{\delta_x^2 K(y)}{\ell(x)^2}u(y)\dd y,
        \end{equation}
        as $x \to 0$.
        Moreover, the square bracket is nonnegative and satisfies
        \begin{equation}
            \label{eq:logarithmicAuxiliaryIntegralLimit}
            \lim_{x \to 0}\int_x^{\nu} \frac{\delta_x^2 K(y)\ell(y)}{\ell(x)^2}\dd y=\frac{1}{2},
        \end{equation}
        while the final term admits the bound
        \begin{equation}
            \label{eq:boundOnRemainderTermInLogarithmicCase}
            0 \leq \int_\nu^\infty\frac{\delta_x^2 K(y)}{\ell(x)^2}u(y)\dd y \leq o(1)
        \end{equation}
        as $x \to 0$.
        \begin{proof}
            Dividing each side of \eqref{eq:condensedFormulationEven} by $\ell(x)^2$, we find
            \begin{equation}
                \label{eq:onTheWayToSimplerLogarithmicIdentity}
                (1+n(u(x)))g(x)^2=\frac{1}{\ell(x)^2}\brak[\Bigg]{\int_0^x+\int_x^{\nu}+\int_\nu^\infty}\delta_x^2 K(y)u(y)\dd y.
            \end{equation}
            where, since $\lim_{x\to0}n(u(x))=0$ by \Cref{ass:nonlinearity,ass:solutionU}, the left-hand side is indeed like that of \eqref{eq:asymptoticEquationForLogarithmicG}. As for the right-hand side, notice that
            \[
                \frac{1}{\ell(x)^2}\int_0^x \delta_x^2K(y) u(y)\dd y = \frac{\rho(x)}{\log(1/x)}g(x)
            \]
            when $\rho$ is defined through
            \[
                \rho(x) \ceq \frac{1}{xu(x)}\int_0^x \delta_x^2 K(y)u(y)\dd y.
            \]

            We exploit that $u$ is increasing on $(0,\nu]$ to conclude that $\rho$ is \emph{bounded} on this interval. This is because
            \[
                \abs{\rho(x)} \leq \frac{1}{x}\int_0^x \abs{\delta_x^2 K(y)}\dd y\leq \int_0^1 \abs{\Lambda(\tau)}\dd \tau + \nu \norm{R''}_{L^1},
            \]
            by \eqref{eq:identityForKernelLogarithmicCase} and \Cref{lem:remainderMayBeDisregarded}. In particular, the first term on right-hand side of \eqref{eq:onTheWayToSimplerLogarithmicIdentity} is $o(1) g(x)$, and after using $u(y)=\ell(y)g(y)$ for the second term, we obtain the right-hand side of \eqref{eq:asymptoticEquationForLogarithmicG}.

            Next, the nonnegativity of the expression inside the square bracket in \eqref{eq:asymptoticEquationForLogarithmicG} is an immediate consequence of the first part of \Cref{lem:secondDifferenceProperties}. To prove \eqref{eq:logarithmicAuxiliaryIntegralLimit}, we use \Cref{lem:remainderMayBeDisregarded} and the boundedness of $\ell$ on $(0,\nu]$  to conclude that $\int_0^\nu \delta_x^2 R(y)\ell(y)\dd y=O(x^2)$. Thus
            \[
                \lim_{x \to 0}{\int_x^\nu \frac{\delta_x^2 K(y) \ell(y)}{\ell(x)^2}\dd y}= \lim_{x \to 0}{\int_1^{\nu / x}} \frac{\Lambda(\tau) \ell(\tau x)}{\ell(x)^2} x\dd \tau
            \]
            from \eqref{eq:identityForKernelLogarithmicCase} and the change of variables $\tau \mapsto \tau x$. By simplifying the integrand, setting $z=\nu/x$, and splitting the integral, this last limit is equal to
            \[
                \lim_{z \to \infty}\brak[\Bigg]{ \frac{\int_1^{z} \Lambda(\tau) \tau\dd \tau }{\log(z/\nu)}+\frac{\int_1^{z} \Lambda(\tau) \ell(\tau)\dd \tau }{\log(z/\nu)^2}}=\lim_{z \to \infty}{\Lambda(z)z^2\brak[\Bigg]{1 -\frac{ \log(z)}{2\log(z/\nu)}}}=\frac{1}{2}.
            \]

            The first equality follows from an application of L'Hôpital's rule to each of the two terms, while the second follows from the observation that $\Lambda(\tau)=1/\tau^2+O(1/\tau^3)$ as $\tau \to \infty$. The latter can be seen directly from its definition in \eqref{eq:secondDifferenceAdditivity}. Finally, the bound in \eqref{eq:boundOnRemainderTermInLogarithmicCase} follows from \Cref{lem:secondDifferenceProperties} and $u$ being nonnegative and bounded.
        \end{proof}
    \end{lemma}

    We are ready to prove the first limit of \hyperref[thm:mainTheorem]{Main Theorem}.
    \begin{proposition}
        \label{prop:logarithmicAsymptoticBehavior}
        Under \Cref{ass:integrabilitySymmetryAndConvexity,ass:nonlinearity,ass:solutionU}, the solution enjoys the limit
        \begin{equation}
            \label{eq:logarithmicLimit}
            \lim_{x \to 0} \frac{u(x)}{x\log(1/x)} = \frac{1}{2}
        \end{equation}
        when $K$ has a logarithmic singularity.
    \end{proposition}

    \begin{proof}
        With $g$ as in \eqref{eq:logarithmicgDefinition}, our strategy is to prove that
        \begin{equation}
            \label{eq:inequalitiesForLimInfAndLimSupOfLogarithmicG}
            \limsup_{x \to 0}{g(x)}\eqc M\leq \tfrac{1}{2}\leq m \ceq \liminf_{x \to 0}{g(x)},
        \end{equation}
        which clearly implies the desired limit.

        We first prove that $m \geq 1/2$. The function $\ul{g}$ defined by
        \begin{equation}
            \label{eq:ulg}
            \ul{g}(x) \ceq \min_{y \in [x,\nu]}{g(y)}
        \end{equation}
        is nondecreasing on $(0,\nu]$, and we find
        \begin{equation}
            \label{eq:inequalityAlongAlambdaLogarithmic}
            g(x)^2 \geq o(1)g(x) + \ul{g}(x)\parn*{\frac{1}{2}+o(1)}
        \end{equation}
        as $x \to 0$ from \eqref{eq:asymptoticEquationForLogarithmicG}. Here we have used both \eqref{eq:logarithmicAuxiliaryIntegralLimit}, and the nonnegativity of the square bracket and the final term.

        Choose now a sequence $\brac{x_k}_{k \in \N} \subset (0,\nu]$ realising $m$. Assuming, for the sake of contradiction, that $m=0$, we may specifically ensure that $g = \ul{g}$ along this sequence. This is possible by positivity and continuity of $g$ on $(0,\nu]$. Then \eqref{eq:inequalityAlongAlambdaLogarithmic} yields
        \[
            g(x_k) \geq \frac{1}{2} + o(1)
        \]
        as $k \to \infty$, after division by $g(x_k)$. Thus, in fact, $m > 0$, and we instead arrive at
        \[
            m^2 \geq \frac{1}{2}\inf_{y \in (0,\nu]}{g(y)}
        \]
        from \eqref{eq:inequalityAlongAlambdaLogarithmic}. Taking the limit $\nu \to 0$, we conclude that $m \geq 1/2$.

        For $M$, we similarly define
        \[
            \ol{g}(x) \ceq \max_{y \in [x,\nu]}{g(y)},
        \]
        which is nonincreasing on $(0,\nu]$, and find
        \begin{equation}
            \label{eq:upperAsymptoticEquationLogarithmic}
            g(x)^2 \leq o(1)g(x) + \ol{g}(x)\parn*{\frac{1}{2}+o(1)}
        \end{equation}
        from \eqref{eq:asymptoticEquationForLogarithmicG}. The last term can no longer be discarded, but can be combined with the first term. This is because of \eqref{eq:boundOnRemainderTermInLogarithmicCase}, and the fact that $m > 0$ entails that $1/g$ is bounded on $(0,\nu]$. Choosing a realising sequence for $M$ in an analogous way, we find $M < \infty$ and
        \[
            M^2 \leq \frac{1}{2} \sup_{y \in (0,\nu]}{g(y)},
        \]
        whence $M \leq 1/2$, after taking the limit $\nu \to 0$.
    \end{proof}
    \subsection{The limit for the derivative}
        \label{ssec:derivativeLimitLogarithmic}
        We move on to proving the second limit of \hyperref[thm:mainTheorem]{Main Theorem}. Analogously to the function $g$ in \eqref{eq:logarithmicgDefinition}, we introduce the quotient
        \begin{equation}
            \label{eq:logarithmichDefinition}
            h(x) \ceq \frac{u'(x)}{L(x)}=\frac{u'(x)}{\log(1/x)},
        \end{equation}
        which is well defined on $(0,\nu]$; where $u'$ is also continuous and nonnegative. By L'Hôpital's rule, a limit of $h$ at zero would immediately imply a limit for $g$. Perhaps curiously, we will go the other way; that is, prove the limit of $h$ by exploiting the already established limit for $g$ in \Cref{prop:logarithmicAsymptoticBehavior}.

        We also introduce -- for notational convenience -- the function
        \begin{equation}
            \label{eq:definitionOfPsiForLimitOfLogarithmicDerivative}
            \Psi(\tau) \ceq \log{\abs*{\frac{1+\tau}{1-\tau}}},
        \end{equation}
        which will serve a similar role to that of $\Lambda$ from \eqref{eq:secondDifferenceAdditivity}. It is positive on $\R^+$ and appears in the relation
        \begin{equation}
            \label{eq:identityForLogarithmicDerivative}
            \delta_{2x}K(\tau x) = \delta_{2x}L(\tau x) + \delta_{2x}R(\tau x) = -\Psi(\tau) + \delta_{2x}R(\tau x)
        \end{equation}
        where $\delta_{2x}f= f(\cdot+x)-f(\cdot-x)$ denotes a first-order central difference.

        \begin{lemma}
            With $h$ and $\Psi$ as defined in \eqref{eq:logarithmichDefinition} and \eqref{eq:definitionOfPsiForLimitOfLogarithmicDerivative} respectively, we have the asymptotic equation
            \begin{equation}
                \label{eq:asymptoticEquationForDerivative}
                (1+o(1))h(x) = \frac{1}{2} +  \int_0^{2} \brak[\Bigg]{\frac{\Psi(\tau)L(\tau x)}{L(x)^2}}h(\tau x)\dd \tau + o(1),
            \end{equation}
            as $x\to0$. Moreover, the expression inside the square brackets satisfies
            \begin{equation}
                \label{eq:derivativeLowerIntegralAsymptotics}
                \lim_{x\to0}\int_0^{2} \frac{\Psi(\tau)L(\tau x)}{L(x)^2}\dd \tau = 0,
            \end{equation}
            and is positive for small $x>0$.
        \end{lemma}
        \begin{proof}
            Recalling the antiderivative $\mc{K}$ from \eqref{eq:K_antiderivative}, the equation takes the form
            \[
                (1+n(u(x)))u(x)^2 = \delta_x^2\mc{K}(\nu)u(\nu) - \int_0^\nu \delta_x^2 \mc{K}(y)u'(y)\dd y + \int_\nu^\infty \delta_x^2 K(y) u(y) \dd y
            \]
            after integrating by parts on the right-hand side of \eqref{eq:condensedFormulationEven}. Subsequently differentiating, we get
            \begin{equation}
                \label{eq:condensedFormulationEvenDifferentiated}
                2 u(x) u'(x) \parn*{1 + \tilde{n}(u(x))} = \delta_{2x}K(\nu)u(\nu) - \int_0^\nu \delta_{2x}K(y)u'(y)\dd y + \int_\nu^\infty \delta_{2x} K'(y) u(y)\dd y
            \end{equation}
            for $x \in (0,\nu)$, where $\tilde{n}(t)\ceq n(t)+ \frac{1}{2}tn'(t)$. This computation is justifiable because of \Cref{ass:solutionU,ass:nonlinearity,ass:integrabilitySymmetryAndConvexity}.

            Using the definition of $h$ and $L$ from \eqref{eq:logarithmichDefinition}, the left-hand side of \eqref{eq:condensedFormulationEvenDifferentiated} can be written
            \begin{equation}
                \label{eq:hEquationLHS}
                2u(x)u'(x)(1+\tilde{n}(u(x))) = xL(x)^2 h(x)\parn*{1+o(1)}
            \end{equation}
            as $x \to 0$, where we have applied \Cref{prop:logarithmicAsymptoticBehavior} and the properties of $n$ from \Cref{ass:nonlinearity}. In particular, this motivates dividing \eqref{eq:condensedFormulationEvenDifferentiated} by $xL(x)^2$. We investigate each term on the right-hand side separately:

            In the first term, we have $\delta_{2x}K(\nu) \leq 0$ by monotonicity of $K$ on $\R^+$, so
            \[
                \delta_{2x}K(\nu) \norm{u}_{L^\infty}\leq \delta_{2x} K(\nu) u(\nu) \leq 0
            \]
            for all $x \in (0,\nu)$. Concerning the tail, we see that convexity of $K$ on $\R^+$ implies that $\delta_{2x}K'(y) \geq 0$ for all $0 < x < y$, and thus
            \[
                0 \leq \int_\nu^\infty \delta_{2x} K'(y)u(y) \dd y \leq \norm{u}_{L^\infty} \int_\nu^\infty \delta_{2x} K'(y)\dd y = -\norm{u}_{L^\infty}\delta_{2x}K(\nu)
            \]
            when $x \in (0,\nu)$. Combined, we therefore have
            \[
                \abs*{\delta_{2x}K(\nu)u(\nu) + \int_\nu^\infty \delta_{2x}K'(y)u(y)\dd y} \leq - \norm{u}_{L^\infty}\delta_{2x}K(\nu) \leq -2\norm{u}_{L^\infty}xK'(\nu-x)
            \]
            on $(0,\nu)$, where the final inequality again is due to the convexity of $K$. Consequently,
            \begin{equation}
                \label{eq:hEquationRHSremainder}
                \frac{1}{xL(x)^2}\parn*{\delta_{2x}K(\nu)u(\nu) + \int_\nu^\infty \delta_{2x}K'(y)u(y)\dd y} = o (1)
            \end{equation}
            as $x \to 0$.

            Turning to the final, evidently dominant, term on the right-hand side of \eqref{eq:condensedFormulationEvenDifferentiated}, we see that
            \[
                -\int_0^\nu \delta_{2x}K(y)u'(y)\dd y = x\int_0^{\nu / x} \Psi(\tau) L(\tau x)h'(\tau x)\dd \tau - \int_0^\nu \delta_{2x}R(y)u'(y)\dd y
            \]
            by \eqref{eq:identityForLogarithmicDerivative} and \eqref{eq:logarithmichDefinition}. We have also made the change of variables $\tau \mapsto \tau x$ in the first integral. Since
            \[
                \abs*{\int_0^\nu \delta_{2x}R(y) u'(y) \dd y} \leq \abs{\delta_{2x}R(y)}u(\nu) \leq 2x \norm{R''}_{L^1}\norm{u}_{L^\infty}
            \]
            by \Cref{ass:integrabilitySymmetryAndConvexity}, we arrive at
            \begin{equation}
                \label{eq:hEquationRHSdominant}
                -\frac{1}{xL(x)^2}\int_0^\nu \delta_{2x}K(y)u'(y)\dd y = \int_0^{\nu / x} \brak*{\frac{\Psi(\tau)L(\tau x)}{L(x)^2}}h(\tau x)\dd \tau + o(1)
            \end{equation}
            as $x \to 0$.

            Thus, dividing \eqref{eq:condensedFormulationEvenDifferentiated} by $xL(x)^2$; followed by inserting \eqref{eq:hEquationLHS}, \eqref{eq:hEquationRHSremainder}, and \eqref{eq:hEquationRHSdominant}; we obtain
            \[
                (1+o(1))h(x) = \int_0^{\nu/x} \brak*{\frac{\Psi(\tau)L(\tau x)}{L(x)^2}}h(\tau x)\dd \tau + o(1)
            \]
            as $x \to 0$. The expression inside the square brackets is clearly nonnegative for $0 <x\leq \nu \ll 1$, and the auxiliary limit \eqref{eq:derivativeLowerIntegralAsymptotics} follows directly from integrability of $\Psi(\tau)$ and $\Psi(\tau)L(\tau)$ on $(0,2)$. The proof will therefore be complete once the limit
            \begin{equation}
                \label{eq:derivativeUpperIntegralAsymptotics}
                \lim_{x \to 0}{\int_2^{\nu/x} \brak*{\frac{\Psi(\tau)L(\tau x)}{L(x)^2}} h(\tau x) \dd \tau } = \frac{1}{2}.
            \end{equation}
            is established.

            To demonstrate this limit, we first argue that, for each fixed $\delta \in(0,\nu)$, we have
            \begin{align}
                \int_{\delta / x}^{\nu / x} \brak*{\frac{\Psi(\tau)L(\tau x)}{L(x)^2}} h(\tau x) \dd \tau & = o(1) \label{eq:thirdAuxiliaryLimitUsedToProveTheFirst}
                \intertext{and}
                \int_{2}^{\delta / x} \frac{\Psi'(\tau)\tau L(\tau x)}{L(x)^2} \dd \tau                   & = -1 + o(1) \label{eq:fourthAuxiliaryLimitUsedToProveTheFirst}
            \end{align}
            as $x \to 0$. Indeed, assuming $x$ is small enough for $\delta/x\geq 2$ to hold, we get that
            \[
                \Psi(\tau)L(\tau x) = \log\parn*{ 1 + \frac{2}{\tau-1}}\log\parn*{\frac{1}{\tau x}} \leq 2\frac{1}{\tau - 1} \log\parn*{\frac{1}{\delta}}\lesssim_\delta \frac{1}{\tau}
            \]
            for all $\tau\geq \delta/x$. Using this, and the fact that $h$ is bounded on $[\delta,\nu]$ (by continuity), we obtain \eqref{eq:thirdAuxiliaryLimitUsedToProveTheFirst}. The second limit, found in \eqref{eq:fourthAuxiliaryLimitUsedToProveTheFirst}, follows from an argument very similar to the one we used to prove \eqref{eq:logarithmicAuxiliaryIntegralLimit}.

            Since $u(x) = xL(x)g(x)$ and $u'(x) = L(x)h(x)$ by definition, see \eqref{eq:logarithmicgDefinition} and \eqref{eq:logarithmichDefinition}, we may use integration by parts to compute that
            \begin{equation}\label{eq:theComputationShowingHowHRelatesToG}
                \begin{aligned}
                    \int_2^{\delta/x} \brak*{\frac{\Psi(\tau)L(\tau x)}{L(x)^2}}h(\tau x) \dd \tau & = \int_2^{\delta/x} \frac{\Psi(\tau)u'(\tau x)}{L(x)^2} \dd \tau                                                                \\
                    & =\frac{\Psi(\delta/x)u(\delta) - \Psi(2)u(2x)}{xL(x)^2} - \int_2^{\delta/x}  \frac{\Psi'(\tau)u(\tau x)}{xL(x)^2} \dd \tau \\
                    & = o(1) -\int_2^{\delta/x} \brak*{\frac{\Psi'(\tau) \tau L(\tau x)}{L(x)^2}}g(\tau x) \dd \tau
                \end{aligned}
            \end{equation}
            as $x\to0$. On the second line, we deal with the first term by using that $\Psi(\delta/x)=O(x)$ and $u(2x)=O(xL(x))$, as $x\to 0$. The latter of the two is a consequence of \Cref{prop:logarithmicAsymptoticBehavior}.

            Adding \eqref{eq:thirdAuxiliaryLimitUsedToProveTheFirst} and \eqref{eq:theComputationShowingHowHRelatesToG} together, and subtracting $1/2$ from each side, we find that
            \begin{equation}\label{eq:finalEquationForTheHEquation}
                \int_2^{\nu/x}\brak[\Bigg]{\frac{\Psi(\tau)L(\tau x)}{L(x)^2}} h(\tau x) \dd \tau-\frac{1}{2}
                =o(1) - \int_2^{\delta/x}\brak[\Bigg]{\frac{\Psi'(\tau)\tau L(\tau x)}{L(x)^2}}\parn[\bigg]{g(\tau x)-\frac{1}{2}} \dd \tau
            \end{equation}
            upon using \eqref{eq:fourthAuxiliaryLimitUsedToProveTheFirst}. Exploiting that the integrand in \eqref{eq:fourthAuxiliaryLimitUsedToProveTheFirst} is single-signed, we can conclude from \eqref{eq:finalEquationForTheHEquation} that
            \begin{equation*}
                \limsup_{x\to0}{\abs[\Bigg]{\int_2^{\nu/x}\brak[\Bigg]{\frac{\Psi(\tau)L(\tau x)}{L(x)^2}} h(\tau x) \dd \tau-\frac{1}{2}}}\leq \sup_{y\in(0,\delta]}{\abs[\bigg]{g(y)-\frac{1}{2}}},
            \end{equation*}
            for every $0 < \delta < \nu$. Thus \eqref{eq:derivativeUpperIntegralAsymptotics}, and hence \eqref{eq:asymptoticEquationForDerivative}, follows by \Cref{prop:logarithmicAsymptoticBehavior}.
        \end{proof}

        We may now prove the desired limit for the derivative in \hyperref[thm:mainTheorem]{Main Theorem}.

        \begin{proposition}
            \label{prop:logarithmicDerivativeAsymptoticBehavior}
            Under \Cref{ass:integrabilitySymmetryAndConvexity,ass:nonlinearity,ass:solutionU}, the derivative of the solution enjoys the limit
            \[
                \lim_{x \to 0}\frac{u'(x)}{\log(1/x)}=\frac{1}{2}
            \]
            when $K$ has a logarithmic singularity.
        \end{proposition}
        \begin{proof}
            With $h$ as in \eqref{eq:logarithmichDefinition}, the result follows immediately from \eqref{eq:asymptoticEquationForDerivative} and \eqref{eq:derivativeLowerIntegralAsymptotics}; provided we are able to show that $h$ is bounded near the origin. We know that $h$ is nonnegative on $(0,\nu]$, so it is sufficient to prove that $h$ is bounded \emph{above} on this set. For the sake of contradiction, suppose that this is not the case; which by continuity of $h$ necessitates blow-up at the origin. As a result, the set
            \[
                A \ceq \brac[\Big]{x \in (0,\nu] : h(x) = \max_{z\in[x,\nu]}h(z)}
            \]
            of points where $h$ is larger than subsequent values must have the origin as an accumulation point. Furthermore, the limit
            \begin{equation}
                \label{eq:limitOfHRestrictedToAIsInfinity}
                \lim_{\substack{x \to 0\\ x \in A}}{h(x)} = \infty
            \end{equation}
            must also hold.

            Observing that $\nu \in A$, we see that the intersection $[x,\nu]\cap A$ is nonempty and closed for any $x\in(0,\nu]$. In particular, the point
            \[
                \ol{x} \ceq \min([x,\nu] \cap A),
            \]
            exists, and enjoys the property
            \begin{equation}
                \label{eq:leastUpperBoundInAProperty}
                h(\ol{x})=\max_{z\in[x,\nu]}h(z).
            \end{equation}
            for every $x \in (0,\nu]$. For convenience, we define the accompanying scaling factor
            \[
                \tau_x\ceq x/\ol{x} \in (0,1]
            \]
            for each $x\in(0,\nu]$. Note also that
            \begin{equation}
                \label{eq:xBarIsSmall}
                \ol{x}= o(1)
            \end{equation}
            as $x \to 0$, by the aforementioned fact that $A$ admits zero as an accumulation point.

            By differentiating, one sees that that $x \mapsto xL(x)^2$ is increasing on $(0,e^{-2})$, which we may assume entirely contains $(0,\nu]$. From this, we obtain
            \[
                \frac{1}{\tau_x L(x)^2} = \frac{\ol{x}}{xL(x)^2} \geq \frac{1}{L(\ol{x})^2},
            \]
            which will be exploited in our next calculation: The change of variables $\tau \mapsto \tau / \tau_x$ yields
            \begin{equation}
                \label{eq:theFirstSingularPartMustBeLargeForSmallerX}
                \begin{aligned}
                    \int_0^1 \brak[\Bigg]{\frac{\Psi(\tau)L(\tau x)}{L(x)^2}}h(\tau x) \dd \tau & = \int_0^{\tau_x} \brak*{\frac{\Psi(\tau/\tau_x)L(\tau \ol{x})}{\tau_x L(x)^2}}h(\tau \ol{x})\dd \tau \\
                                                                                                & \geq \int_0^{\tau_x} \brak*{\frac{\Psi(\tau)L(\tau \ol{x})}{L(\ol{x})^2}}h(\tau \ol{x})\dd \tau,
                \end{aligned}
            \end{equation}
            where we have also used that $\Psi$ is positive, and increasing on $(0,1)$. This can be seen directly from its definition in \eqref{eq:definitionOfPsiForLimitOfLogarithmicDerivative}.

            As we shall see, \eqref{eq:theFirstSingularPartMustBeLargeForSmallerX} actually implies that $h(x)$ is comparable to $h(\ol{x})$. Taking the difference of \eqref{eq:asymptoticEquationForDerivative} evaluated at $x$ and $\ol{x}$, respectively, we get
            \begin{multline*}
                (1+o(1))h(x) - (1 + o(1))h(\ol{x})\\
                = \int_0^{2} \brak[\Bigg]{\frac{\Psi(\tau)L(\tau x)}{L(x)^2}}h(\tau x)\dd \tau - \int_0^{2} \brak[\Bigg]{\frac{\Psi(\tau)L(\tau \ol{x})}{L(\ol{x})^2}}h(\tau \ol{x})\dd \tau + o(1)
            \end{multline*}
            as $x \to 0$, after using \eqref{eq:xBarIsSmall}. On the right-hand side,
            \[
                \int_0^{2} \brak[\Bigg]{\frac{\Psi(\tau)L(\tau x)}{L(x)^2}}h(\tau x)\dd \tau \geq \int_0^{\tau_x} \brak*{\frac{\Psi(\tau)L(\tau \ol{x})}{L(\ol{x})^2}}h(\tau \ol{x})\dd \tau
            \]
            by \eqref{eq:theFirstSingularPartMustBeLargeForSmallerX}, and positivity of the integrand. Thus
            \begin{equation}
                \label{eq:hBarAndHAsymptotic}
                \begin{aligned}
                    (1+o(1))h(x) - (1 + o(1))h(\ol{x}) & \geq - \int_{\tau_x}^2 \brak*{\frac{\Psi(\tau)L(\tau \ol{x})}{L(\ol{x})^2}}h(\tau \ol{x})\dd \tau + o(1) \\
                                                       & \geq o(1)h(\ol{x}) + o(1)
                \end{aligned}
            \end{equation}
            as $x\to 0$, in view of \eqref{eq:xBarIsSmall}, \eqref{eq:leastUpperBoundInAProperty}, and \eqref{eq:derivativeLowerIntegralAsymptotics}.

            As a consequence of \eqref{eq:hBarAndHAsymptotic}, we conclude that
            \[
                \liminf_{x \to 0}{\frac{h(x)}{h(\ol{x})}} \geq 1,
            \]
            which, since $\lim_{x \to 0}{h(\ol{x})} = \infty$ by \eqref{eq:limitOfHRestrictedToAIsInfinity}, implies that
            \[
                \lim_{x \to 0}{h(x)} = \infty
            \]
            holds. This leads to our contradiction: With $g$ as in \eqref{eq:logarithmicgDefinition}, we see through \eqref{eq:logarithmicLimit} and integration by parts that
            \[
                \frac{1}{x} \int_0^x h(y)\dd y = g(x) - \frac{1}{x} \int_0^x \frac{g(y)}{\log(1/y)}\dd y \to \frac{1}{2}
            \]
            as $x\to0$. For this to be the case, we must necessarily have $\liminf_{x\to0}h(x)<\infty$; contradicting what we just demonstrated. In conclusion, $h$ is bounded on $(0,\nu]$, and the proof is complete.
        \end{proof}

    \subsection{The bidirectional Whitham equation} \label{sec:applicationForTheHighestWavesOfBidirectionalWhithamEquation}
        By inserting the steady-wave ansatz $\phi(t,x) = \varphi(x-ct)$ and $(t,x)\mapsto v(x-ct)$ into \eqref{eq:bidirectionalWhitham} and integrating, the time-independent Whitham--Boussinesq system
        \begin{equation}\label{eq:WB}
            \begin{aligned}
                -c\varphi + K_B * v + \varphi v & = 0, \\
                -cv + \varphi + v^2/2           & =0,
            \end{aligned}
        \end{equation}
        is obtained. The constants of integration have been set to zero in order to match the setting of \cite{Ehrnstroem19Existence}. By subsequently eliminating $\varphi$, we find the \emph{steady bidirectional Whitham equation}
        \begin{equation}
            \label{eq:steadyBidirectionalWhitham}
            K_B*v = v(c-v)(c-v/2)
        \end{equation}
        for $v$. Given a solution to \eqref{eq:steadyBidirectionalWhitham}, the associated $\varphi$ can easily be recovered through the second equation in \eqref{eq:WB}.

        We see that even if \eqref{eq:steadyBidirectionalWhitham} arose from a system, it is of the exact same type as \eqref{eq:generalEquationSteady}. Repeating the procedure in \Cref{sec:overview}, we first discern that the right-hand side of \eqref{eq:steadyBidirectionalWhitham} increases to the left of a local maximum at $v =(1-1/\sqrt{3})c$. If $v$ is even, and assumes this value at the origin, then
        \begin{equation}
            \label{eq:rescaledBidirectionalWhithamVariable}
            u \ceq \frac{\sqrt{3}\pi c}{2}\parn*{\parn*{1-\frac{1}{\sqrt{3}}}c-v}
        \end{equation}
        satisfies the equation
        \[
            \parn*{1+\frac{2}{3\pi c^2}u(x)}u(x)^2 = \int_0^\infty \delta_x^2 (\pi K_B)(y)u(y)\dd y,
        \]
        which is precisely of the form \eqref{eq:condensedFormulationEven}. Moreover, \Cref{ass:nonlinearity} holds trivially, and the formula
        \[
            K_B(x) = \frac{1}{\pi}\log\parn*{\coth\parn*{\frac{\pi \abs{x}}{4}}}
        \]
        from \eqref{eq:bidirectionalWhithamSymbol} and \cite{Oberhettinger90Tables}*{I.7.37} shows that \Cref{ass:integrabilitySymmetryAndConvexity} is satisfied with $S=L$.

        In \cite{Ehrnstroem19Existence}*{Theorem 5.9}, the existence of a limiting $2\pi$-periodic solution $(v,c)$ of \eqref{eq:steadyBidirectionalWhitham} is established. This solution is even, assumes $v(0) = (1-1/\sqrt{3})c$ at the crest, decreases on the half-period $[0,\pi]$, and is smooth on $(0,2\pi)$. In particular, \Cref{ass:solutionU} holds both for this solution, and for similar solutions with a different period. The hypotheses of \Cref{prop:logarithmicAsymptoticBehavior,prop:logarithmicDerivativeAsymptoticBehavior} are therefore satisfied for the rescaled variable in \eqref{eq:rescaledBidirectionalWhithamVariable}. From this, we may deduce the asymptotic behaviour of $v$, and in turn that of $\varphi$.

        \begin{corollary}[Asymptotic behaviour of highest waves]
            \label{cor:bidirectionalWhithamAsymptotic}
            Let $v$ be a solution of the steady bidirectional Whitham equation \eqref{eq:steadyBidirectionalWhitham} that is even, assumes $v(0) = (1-1/\sqrt{3})c$, and is smooth and decreasing on a nonempty interval $(0,\nu)$. Then
            \begin{equation}
                \label{eq:bidirectionalWhithamAsymptotic}
                \begin{aligned}
                    v(x)       & = \parn*{1-\frac{1}{\sqrt{3}}}c - \parn*{\frac{1}{\sqrt{3}\pi c}+o(1)}x\log(1/x) \\
                    \varphi(x) & = \frac{c^2}{3} - \parn*{\frac{1}{3\pi}+o(1)}x\log(1/x)
                \end{aligned}
            \end{equation}
            as $x \searrow 0$, with $\varphi$ as described after \eqref{eq:steadyBidirectionalWhitham}. Moreover, one also has
            \begin{equation}
                \label{eq:bidirectionalWhithamDerivativeAsymptotic}
                \begin{aligned}
                    v'(x)       & =-\parn*{\frac{1}{\sqrt{3}\pi c}+o(1)}\log(1/x) \\
                    \varphi'(x) & = - \parn*{\frac{1}{3\pi}+o(1)}\log(1/x)
                \end{aligned}
            \end{equation}
            as $x \searrow 0$.
        \end{corollary}

        The authors of \cite{Ehrnstroem19Existence} also pose a natural question about the \emph{global} regularity of these waves: What is a reasonable function space that can capture the kind of asymptotic behaviour in \eqref{eq:bidirectionalWhithamAsymptotic} in an optimal way? A sensible candidate is the space of log-Lipschitz functions \cite{Edmunds00Embeddings}. This space appears, for instance, in critical Sobolev embeddings, and as a simple example of a class of non-Lipschitz right-hand sides for which the Osgood criterion \cite{Fjordholm18Sharp} for the Picard--Lindelöf theorem holds.

        This global regularity is not a direct consequence of the local behaviour in \eqref{eq:bidirectionalWhithamAsymptotic}. Oscillations may, even under additional assumptions of monotonicity and smoothness, cause the estimates to blow up in the limit. We will show that the highest waves indeed are log-Lipschitz by combining \eqref{eq:bidirectionalWhithamDerivativeAsymptotic} with fairly straightforward bounds. To get the result, it is advantageous to introduce the concept of a modulus of continuity, commonly used in approximation theory.

        We shall say that $\map{\omega}{\R_0^+}{\R_0^+}$ is a \emph{modulus of continuity} if it is increasing, concave, continuous, and vanishes at the origin. Any function $\map{f}{I}{\R}$ is then said to admit $\omega$ as a modulus of continuity if
        \[
            \abs{f(x)-f(y)}\leq \omega(\abs{x-y})
        \]
        for all $x,y \in I$. The following simple lemma is ours, but is very likely known in some form in the literature. It can be viewed as a kind of L'Hôpital's rule for moduli of continuity.

        \begin{lemma}
            \label{lem:localModulus}
            Suppose that $f$ is absolutely continuous on an open interval $I \ni 0$, and that
            \begin{equation}
                \label{eq:localModulusLimSupFinite}
                \elimsup_{t \to 0}{\frac{\abs{f'(t)}}{\omega'(\abs{t})}} < \infty
            \end{equation}
            for a modulus of continuity $\omega$. Then there are $M,\delta > 0$ such that $f$ admits $M\omega$ as a modulus of continuity on $(-\delta,\delta)$.
        \end{lemma}
        \begin{proof}
            Note that $\omega$ is necessarily locally absolutely continuous. Due to \eqref{eq:localModulusLimSupFinite}, we are able to find $M,\delta > 0$ such that
            \[
                \abs{f'(t)} \leq \frac{M}{2}\omega'(\abs{t})
            \]
            for a.e. $t \in (-\delta,\delta)$. It follows that
            \[
                \abs{f(y)-f(x)} = \abs*{\int_x^y f'(t)\dd t} \leq \frac{M}{2}\int_x^y \omega'(\abs{t})\dd t = \frac{M}{2}\brak[\big]{\omega(\abs{t})\sgn(t)}_x^y
            \]
            for all $x \leq y \in (-\delta,\delta)$.

            Since $\omega$ is concave on $\R_0^+$, and $\omega(0) \geq 0$, it is also subadditive. Thus
            \[
                \brak[\big]{\omega(\abs{t})\sgn(t)}_x^y = \omega(y) - \omega(x) \leq \omega\parn*{y-x} = \omega(\abs{y-x})
            \]
            when $0 \leq x \leq y$, and a similar line of reasoning works for the case $x \leq y \leq 0$. Finally, if $x \leq 0 \leq y$, then
            \[
                \brak[\big]{\omega(\abs{t})\sgn(t)}_x^y = \omega(y)+\omega(-x) \leq \omega(y-x) + \omega(y-x) = 2\omega(\abs{y-x})
            \]
            by monotonicity of $\omega$. This concludes the proof.
        \end{proof}

        It is furthermore straightforward to show that if $f$ admits $M_i \omega$ as a modulus of continuity on an interval $I_i$ for $i = 1,2$, and $I_1 \cap I_2 \neq \varnothing$, then $f$ admits $(M_1 + M_2)\omega$ as a modulus of continuity on $I_1 \cup I_2$. This follows since for any $x \in I_1$ and $y \in I_2$, there is some $z \in I_1 \cap I_2$ between $x$ and $y$, whence
        \begin{equation}
            \label{eq:stitching}
            \abs*{f(y)-f(x)} \leq M_2\omega(\abs{y-z}) + M_1\omega(\abs{z-x})\leq (M_1 + M_2)\omega\parn*{\abs{y-x}},
        \end{equation}
        by monotonicity of $\omega$. We use this to get the following result.

        \begin{theorem}[Global regularity of highest waves]
            Any periodic solution to \eqref{eq:steadyBidirectionalWhitham} satisfying the hypothesis of \Cref{cor:bidirectionalWhithamAsymptotic} around its crests, belongs to the class of log-Lipschitz functions. That is, there is a constant $M > 0$ such that
            \begin{equation}
                \label{eq:globalLogLipschitz}
                \abs{v(x)-v(y)} \leq M\abs{x-y}\log\parn*{1+\frac{1}{\abs{x-y}}}
            \end{equation}
            for all $x,y \in \R$.
        \end{theorem}
        \begin{proof}
            Because of \eqref{eq:bidirectionalWhithamDerivativeAsymptotic}, we can apply \Cref{lem:localModulus} with $\omega(t) \ceq t\log(1+1/t)$ to get \eqref{eq:globalLogLipschitz} in a neighborhood of each crest. Meanwhile, away from crests, the same conclusion holds by the smoothness furnished by \cite{Ehrnstroem19Existence}*{Lemma 4.1}. Since we have a compact domain from periodicity, the stitching argument in \eqref{eq:stitching} enables us to infer that there is a uniform constant $M > 0$ for which \eqref{eq:globalLogLipschitz} holds globally.
        \end{proof}
        \begin{remark}
            For a highest solitary solution, the same conclusion can be reached by combining compactness with a priori decay properties; see for instance \cite{Arnesen22Decay}. \emph{Small} solitary-wave solutions to the Whitham--Boussinesq system \eqref{eq:bidirectionalWhitham} were constructed in \cite{Nilsson19Solitary}, but at present there is no existence result for extreme solutions in the solitary case.
        \end{remark}

\section{Homogeneous kernel}
    \label{sec:homogeneousKernel}
    Like in the previous section, we shall adopt \Cref{ass:integrabilitySymmetryAndConvexity,ass:nonlinearity,ass:solutionU}, but now take the singular part of the kernel $K$ to be $S(x)=H(x)=\abs{x}^{-1/2}$. The same restriction of $x$ to $(0,\nu]$ for some $0 < \nu \ll 1$ will also be made. We will here prove the final two limits of \hyperref[thm:mainTheorem]{Main Theorem}, and so analogously to \eqref{eq:logarithmicgDefinition} define
    \begin{equation}
        \label{eq:homogeneousgDefinition}
        g(x) \ceq \frac{u(x)}{x^{1/2}}
    \end{equation}
    for $x>0$ in this section. We further remind the reader of $\Phi$ from \eqref{eq:homogeneousPhiDefinition}, which appears in the identity
    \begin{equation}
        \label{eq:identityForKernelHomogeneousCase}
        \delta_x^2 K(\tau x) =\delta_x^2 H(\tau x)+\delta_x^2 R(\tau x) =x^{-1/2}\Phi(\tau) + \delta_x^2 R(\tau x)
    \end{equation}
    due to \eqref{eq:secondDifferenceHomogeneity}.

    Understanding the properties of $\Phi$ will clearly be paramount for the calculations in this section, and we therefore start with a lemma listing a few of them. The bounds are certainly not optimal, but sufficient for our purposes. See also \Cref{fig:PhiGraph} in \Cref{sec:overview}.
    \begin{lemma}
        \label{lem:propertiesOfHomogeneousPhi}
        The function $\Phi$ is increasing on the interval $(0,1)$, where it has a unique root $\tau_0 \in (\frac{1}{2},\frac{2}{3})$. In addition, $\Phi$ is positive on $(1,\infty)$,
        \begin{equation}
            \label{eq:PhiIntegrals}
            \int_0^\infty \Phi(\tau) \dd \tau = 0, \quad \int_0^\infty \Phi(\tau)\tau^{1/2} \dd \tau = \frac{\pi}{2},
        \end{equation}
        and
        \begin{equation}
            \label{eq:bBounds}
            \frac{1}{2} < -\int_0^{\tau_0} \Phi(\tau)\tau^{1/2}\dd \tau  < \frac{3}{5}.
        \end{equation}
    \end{lemma}
    \begin{proof}
        The first integral in \eqref{eq:PhiIntegrals} is a trivial computation, while the second explicit integral is simply a special case of \eqref{eq:betaIntegralIdentity}. That $\Phi$ is increasing on $(0,1)$ follows directly from differentiating \eqref{eq:homogeneousPhiDefinition}; and, by explicit evaluation, one further sees that $\Phi(\frac{1}{2})<0<\Phi(\frac{2}{3})$. Hence there is a unique root $\tau_0$ on the interval, which necessarily lies in $(\frac{1}{2},\frac{2}{3})$. The positivity on $(1,\infty)$ follows by the same computation as for $\delta_x^2 K$ in \Cref{lem:secondDifferenceProperties}, using the strict convexity of $H$ on $\R^+$.

        For \eqref{eq:bBounds}, it is easily verified that
        \begin{equation}
            \label{eq:bDefinition}
            - \int_0^t \Phi(\tau)\tau^{1/2}\dd \tau  =2t-\frac{2t^{3/2}}{(1+t)^{1/2}+(1-t)^{1/2}}+\arsinh(t^{1/2})-\arcsin(t^{1/2}),
        \end{equation}
        for all $t\in(0,1)$. Due to the sign-change of $\Phi$ at $\tau = \tau_0$, this integral is maximised there, and lower bounds can be found by evaluation at any other point. In particular, we find
        \[
            - \int_0^{\tau_0} \Phi(\tau)\tau^{1/2}\dd \tau  >  - \int_0^{2/3} \Phi(\tau)\tau^{1/2}\dd \tau  > \frac{1}{2}.
        \]
        by evaluating \eqref{eq:bDefinition} at $2/3$.

        To establish the upper bound, we observe through \eqref{eq:bDefinition} and straightforward algebra that
        \[
            - \int_0^t \Phi(\tau)\tau^{1/2}\dd \tau + t^{3/2}\Phi(t) = \frac{1}{(t^{-1}-1)^{1/2}}-\frac{1}{(t^{-1}+1)^{1/2}} +\arsinh(t^{1/2})-\arcsin(t^{1/2}),
        \]
        and that this expression is increasing on $(0,1)$. Exploiting this, we get
        \[
            - \int_0^{\tau_0} \Phi(\tau)\tau^{1/2}\dd \tau =  - \int_0^{\tau_0} \Phi(\tau)\tau^{1/2}\dd \tau  + \tau_0^{3/2}\Phi(\tau_0)  <\frac{3}{5}
        \]
        after using the fact that $\tau_0$ is a root of $\Phi$, and evaluating at $2/3 > \tau_0$.
    \end{proof}

    We next provide an asymptotic rephrasing of \eqref{eq:condensedFormulationEven} for $g$, analogous to the one provided by \Cref{lem:theLemmaThatProvidesTheAsymptoticEquationForLogarithmicG} in the previous section. A fundamental difference from the logarithmic case is that the contribution from the integral $\int_0^x\delta_x^2K(y)u(y)\dd y$ can no longer be disregarded when passing to the limit. This is unlike \eqref{eq:asymptoticEquationForLogarithmicG}, and makes the subsequent arguments more involved.

    \begin{lemma}
        \label{lem:theLemmaThatProvidesTheAsymptoticEquationForHomogeneousG}
        With $g$ defined as in \eqref{eq:homogeneousgDefinition}, there is a function $\map{\lambda}{(0,1)}{(0,1)}$ so that
        \begin{multline}
            \label{eq:asymptoticEquationForHomogeneousG}
            (1+o(1))g(x)^2\\
            =\parn[\bigg]{\int_{\lambda(x)}^1\Phi(\tau)\dd \tau+o(1)}g(x)+\int_x^{\nu}\brak[\Bigg]{\frac{\delta_x^2K(y)y^{1/2}}{x}}g(y)\dd y+\int_\nu^\infty \frac{\delta_x^2 K(y)}{x}u(y)\dd y
        \end{multline}
        as $x\to0$. Moreover, the square bracket is positive and satisfies
        \begin{equation}
            \label{eq:homogeneousAuxiliaryIntegralLimit}
            \lim_{x \to 0}{\int_x^\nu \frac{\delta_x^2 K(y)y^{1/2}}{x}\dd y}=\int_{1}^\infty \Phi(\tau)\tau^{1/2}\dd \tau,
        \end{equation}
        while the final term admits the bound
        \begin{equation}
            \label{eq:boundOnRemainderTermIHomogeneousCase}
            0 \leq \int_\nu^\infty \frac{\delta_x^2 K(y)}{x}u(y)\dd y \leq o(1)
        \end{equation}
        as $x \to 0$
    \end{lemma}
    \begin{proof}
        Dividing each side of \eqref{eq:condensedFormulationEven} by $x$, we find
        \[
            (1+n(u(x)))g(x)^2 = \frac{1}{x}\brak*{\int_0^x + \int_x^\nu + \int_\nu^\infty}\delta_x^2 K(y)u(y)\dd y,
        \]
        where we observe that the left-hand side is of the same form as in \eqref{eq:asymptoticEquationForHomogeneousG}. On the right-hand side, the third integrals are identical, while in the second we have simply used the definition of $g$ in \eqref{eq:homogeneousgDefinition} to write $u(y) = y^{1/2}g(y)$. The first integral requires more elaboration.

        Recall from \Cref{lem:propertiesOfHomogeneousPhi} that the singular part of $\delta_x^2 K(y)$ changes sign at $\tau_0 x$. As $u$ is increasing and nonnegative on $[0,\nu]$, we may still make use of the second mean value theorem for integrals on the first integral. Explicitly, we are able to conclude that, for every $x \in (0,\nu]$, we have
        \[
            \int_0^x \delta_x^2 K(y)u(y)\dd y = u(x)\int_{\lambda(x)x}^x \delta_x^2 K(y)\dd y.
        \]
        for some $\lambda(x) \in (0,1)$. Here, combining the identity \eqref{eq:identityForKernelHomogeneousCase} with \Cref{lem:remainderMayBeDisregarded}, we further have
        \[
            \int_{\lambda(x)x}^x \delta_x^2 K(y)\dd y = x^{1/2}\int_{\lambda(x)}^1 \Phi(\tau)\dd \tau + O(x^2),
        \]
        as $x \to 0$. Consequently, we find the first term in \eqref{eq:asymptoticEquationForHomogeneousG}.

        The positivity of the expression inside the square brackets in \eqref{eq:asymptoticEquationForHomogeneousG} for $y \in (x,\nu]$ is an immediate corollary of \Cref{lem:secondDifferenceProperties}, while the limit \eqref{eq:homogeneousAuxiliaryIntegralLimit} follows directly from \eqref{eq:identityForKernelHomogeneousCase} and \Cref{lem:remainderMayBeDisregarded}. Finally, by an argument identical to the one used to prove \eqref{eq:boundOnRemainderTermInLogarithmicCase}, we obtain \eqref{eq:boundOnRemainderTermIHomogeneousCase}.
    \end{proof}

    As we have alluded to, applying the arguments in the proof of \Cref{prop:logarithmicAsymptoticBehavior} to \eqref{eq:asymptoticEquationForHomogeneousG} will \emph{not} directly lead us to the desired limit for $g$. Instead, we will derive a system of two inequalities for the limits inferior and superior of $g$ at zero. As will be demonstrated in \Cref{prop:homogeneousAsymptoticBehavior}, these inequalities are in fact sharp enough to ensure the limit for $g$.

    \begin{lemma}\label{lem:theLemmaProvidingTheInequalitiesForLimInfAndLimSupOfHomogeneousG}
        With $g$ as in \eqref{eq:homogeneousgDefinition}, we have
        \[
            m\ceq\liminf_{x\to0}{g(x)} > 0 \quad \text{and} \quad M \ceq \limsup_{x\to0}{g(x)} < \infty,
        \]
        for which the inequalities
        \begin{align}
            \label{eq:refinedInequalityForLargeM}
            M^2 & \leq  m\int_0^{\tau_0} \Phi(\tau)\tau^{1/2}\dd \tau +  M\int_{\tau_0}^\infty \Phi(\tau)\tau^{1/2}\dd \tau,                   \\\label{eq:refinedInequalityForSmallM}
            m^2 & \geq \int_0^{\tau_0} \Phi(\tau)\min\parn[\big]{m,M\tau^{1/2}}\dd \tau + m\int_{\tau_0}^\infty \Phi(\tau) \tau^{1/2}\dd \tau,
        \end{align}
        hold. Here, $\Phi$ is as defined in \eqref{eq:homogeneousPhiDefinition}.
    \end{lemma}

    \begin{remark}
        Compare \eqref{eq:refinedInequalityForLargeM} and \eqref{eq:refinedInequalityForSmallM} with the more symmetric \eqref{eq:symmetric_inequalities} that was covered in \Cref{sec:overview}. Without the refinement of \eqref{eq:refinedInequalityForSmallM} over the second inequality in \eqref{eq:symmetric_inequalities}, the system would be too weak to reach the conclusion of \Cref{prop:homogeneousAsymptoticBehavior}.
    \end{remark}
    \begin{proof}
        We first prove that $m>0$. Proceeding as in \Cref{prop:logarithmicAsymptoticBehavior}, we deduce from \eqref{eq:asymptoticEquationForHomogeneousG} that
        \begin{equation}
            \label{eq:lowerAsymptoticEquationHomogeneous}
            (1+o(1))g(x)^2\geq \parn[\bigg]{\int_{0}^1\Phi(\tau)\dd \tau+o(1)}g(x)+\ul{g}(x)\int_x^{\nu}\brak[\Bigg]{\frac{\delta_x^2K(y)y^{1/2}}{x}}\dd y
        \end{equation}
        as $x \to 0$. Here, $\ul{g}$ is again defined according to \eqref{eq:ulg}. To bound the integral below, we have used the monotonicity of $\Phi$ on $(0,1)$, along with $\int_0^1 \Phi(\tau)\dd \tau < 0$; both from \Cref{lem:propertiesOfHomogeneousPhi}.

        Assuming that $m=0$, we may pick a realising sequence $\{x_k\}_{k \in \N} \subset (0,\nu]$ for $m$, in such a way that $g = \ul{g}$ along the sequence. Then \eqref{eq:lowerAsymptoticEquationHomogeneous} reduces to
        \[
            (1+o(1))g(x_k) \geq \int_0^1 \Phi(\tau) \dd \tau + \int_{x_k}^{\nu}\brak[\Bigg]{\frac{\delta_{x_k}^2K(y)y^{1/2}}{x_k}}\dd y + o(1)
        \]
        as $k \to \infty$, after having divided by $g(x_k)$. Going to the limit, we obtain
        \[
            m \geq \int_0^1 \Phi(\tau)\dd \tau + \int_1^\infty \Phi(\tau)\tau^{1/2}\dd \tau = \int_1^\infty \Phi(\tau)(\tau^{1/2}-1) > 0,
        \]
        where the equality comes from the first integral in \eqref{eq:PhiIntegrals}. Meanwhile, the final inequality stems from positivity of the integrand on $(1,\infty)$, which is part of \Cref{lem:propertiesOfHomogeneousPhi}. Regardless, this is a contradiction, so $m > 0$.

        Similarly, arguing like for \eqref{eq:upperAsymptoticEquationLogarithmic}, one has
        \[
            (1+o(1))g(x)^2\leq \parn[\bigg]{\int_{\tau_0}^1\Phi(\tau)\dd \tau+o(1)}g(x)+\ol{g}(x)\int_x^{\nu}\brak[\Bigg]{\frac{\delta_x^2K(y)y^{1/2}}{x}}\dd y,
        \]
        as $x \to 0$. Assuming that $M = \infty$, we are again able to choose a realising sequence along which $g = \ol{g}$. This results in the contradiction
        \[
            M \leq \int_{\tau_0}^1 \Phi(\tau)\dd \tau + \int_1^\infty \Phi(\tau)\tau^{1/2}\dd \tau < \infty,
        \]
        and so we do in fact have $M < \infty$.

        Armed with the knowledge that $0<m\leq M <\infty$, we may now derive the sharper inequalities \eqref{eq:refinedInequalityForLargeM} and \eqref{eq:refinedInequalityForSmallM}: Knowing that $g$ is bounded, \eqref{eq:asymptoticEquationForHomogeneousG} can be replaced with the simpler
        \begin{equation}
            \label{eq:simpleEquationForHomogeneousG}
            g(x)^2 = \int_0^{\nu/x} \Phi(\tau) \tau^{1/2} g(\tau x)\dd \tau + o(1)
        \end{equation}
        as $x \to 0$, by employing \Cref{lem:secondDifferenceProperties,lem:remainderMayBeDisregarded}. Since we recall from \Cref{lem:propertiesOfHomogeneousPhi} that $\Phi$ is negative on $(0,\tau_0)$, and positive on $(\tau_0,\infty)$, we therefore see that
        \[
            M^2 \leq \parn[\bigg]{\inf_{y \in (0,\nu]}{g(y)}}\int_0^{\tau_0}\Phi(\tau)\tau^{1/2}\dd \tau + \parn[\bigg]{\sup_{y \in (0,\nu]}{g(y)}}\int_{\tau_0}^\infty\Phi(\tau)\tau^{1/2}\dd \tau,
        \]
        which yields \eqref{eq:refinedInequalityForLargeM} in the limit $\nu \to 0$.

        In order to establish \eqref{eq:refinedInequalityForSmallM}, we note that, since $u$ is increasing on $(0,\nu]$, we have
        \[
            \tau^{1/2}g(\tau x)=\frac{u(\tau x)}{x^{1/2}}\leq \frac{u(x)}{x^{1/2}} = g(x)
        \]
        for every $\tau \in(0,1)$ and $x \in(0,\nu]$. Thus
        \[
            \tau^{1/2}g(\tau x) \leq \min\parn[\bigg]{g(x),\tau^{1/2}\parn[\bigg]{\sup_{y \in (0,\nu]}{g(y)}}}
        \]
        for all such $x,\tau$, and the lower bound
        \[
            g(x)^2 \geq \int_0^{\tau_0} \Phi(\tau)\min\parn[\bigg]{g(x),\tau^{1/2}\parn[\bigg]{\sup_{y \in (0,\nu]}{g(y)}}}\dd \tau + \parn[\bigg]{\inf_{y \in (0,\nu]}{g(y)}} \int_{\tau_0}^\infty\Phi(\tau)\tau^{1/2}\dd \tau +o(1)
        \]
        as $x\to 0$, is therefore obtained from \eqref{eq:simpleEquationForHomogeneousG}. Finally, we are left with \eqref{eq:refinedInequalityForSmallM} after first taking the limit along a sequence realising $m$, and subsequently letting $\nu \to 0$.
    \end{proof}
    While the inequalities \eqref{eq:refinedInequalityForLargeM} and \eqref{eq:refinedInequalityForSmallM} are more involved than the corresponding inequalities found in the logarithmic case \eqref{eq:inequalitiesForLimInfAndLimSupOfLogarithmicG}, it just so happens that the \emph{only} point $(m,M)\in \R^+\times \R^+$ that satisfies both \eqref{eq:refinedInequalityForLargeM}, \eqref{eq:refinedInequalityForSmallM}, and $m\leq M$, is the one given by $m=M=\pi/2$. We now prove this, resulting in the third limit of \hyperref[thm:mainTheorem]{Main Theorem}.

    \begin{proposition}
        \label{prop:homogeneousAsymptoticBehavior}
        Under \Cref{ass:integrabilitySymmetryAndConvexity,ass:nonlinearity,ass:solutionU}, the solution enjoys the limit
        \[
            \lim_{x \to 0} \frac{u(x)}{x^{1/2}} = \frac{\pi}{2}
        \]
        when $K$ has a homogeneous singularity.
    \end{proposition}
    \begin{proof}
        With $M$ and $m$ as in \Cref{lem:theLemmaProvidingTheInequalitiesForLimInfAndLimSupOfHomogeneousG}, we first introduce $\sigma\ceq M/m\geq 1$, and rewrite \eqref{eq:refinedInequalityForLargeM} and \eqref{eq:refinedInequalityForSmallM} purely in terms of $\sigma$ and $m$:
        \begin{align}
            \label{eq:mRefinedUpperBoundAfterSubstitutingInSigma}
            m & \leq \sigma^{-2} \int_0^{\tau_0} \Phi(\tau)\tau^{1/2}\dd \tau + \sigma^{-1} \int_{\tau_0}^\infty \Phi(\tau)\tau^{1/2}\dd \tau, \\
            \label{eq:mRefinedLowerBoundAfterSubstitutingInSigma}
            m & \geq \int_0^{\tau_0} \Phi(\tau)\min(1,\sigma \tau^{1/2})\dd \tau + \int_{\tau_0}^\infty \Phi(\tau) \tau^{1/2}\dd \tau.
        \end{align}

        When $\sigma=1$, both right-hand sides read $\int_0^\infty\Phi(y)y^{1/2}\dd y=\pi/2$ by \Cref{lem:propertiesOfHomogeneousPhi}, and so $\pi/2=m = M/\sigma=M$. We will therefore be done once we are able to show that no $m>0$ simultaneously satisfies both \eqref{eq:mRefinedUpperBoundAfterSubstitutingInSigma} and \eqref{eq:mRefinedLowerBoundAfterSubstitutingInSigma} when $\sigma>1$. To that end, we introduce
        \begin{equation}
            \label{eq:fDefinition}
            f(\sigma) \ceq \int_0^{\tau_0} \Phi(\tau)\min(1,\sigma \tau^{1/2})\dd \tau + \sigma^{-2}b + \parn*{\frac{\pi}{2}+b}(1-\sigma^{-1})
        \end{equation}
        for $\sigma \geq 1$, where
        \[
            b\ceq -\int_0^{\tau_0}\Phi(\tau)\tau^{1/2}\dd \tau
        \]
        is a positive constant. As $\int_{\tau_0}^\infty \Phi(\tau)\tau^{1/2}\dd \tau=\pi/2+ b$, it is not difficult to see that $f(\sigma)$ is precisely the right-hand side of \eqref{eq:mRefinedLowerBoundAfterSubstitutingInSigma} \emph{minus} that of \eqref{eq:mRefinedUpperBoundAfterSubstitutingInSigma}. Hence, if we can demonstrate that $f$ is positive on $(1,\infty)$, then there is no simultaneous solution to \eqref{eq:mRefinedUpperBoundAfterSubstitutingInSigma} and \eqref{eq:mRefinedLowerBoundAfterSubstitutingInSigma}; thereby completing the proof

        Since $\Phi$ is negative on $(0,\tau_0)$ by \Cref{lem:propertiesOfHomogeneousPhi}, we may use the trivial inequality $\min(1,\sigma \tau^{1/2}) \leq \sigma \tau^{1/2}$ to see that
        \begin{align*}
            f(\sigma) & \geq -\sigma b +\sigma^{-2}b + \parn*{\frac{\pi}{2}+b}(1-\sigma^{-1}) \\
                      & =(1-\sigma^{-1})\parn*{\frac{\pi}{2}-b(\sigma+\sigma^{-1})}
        \end{align*}
        for all $\sigma \geq 1$. Note that the first of the two factors is positive for all $\sigma > 1$, while the second factor is a decreasing function of $\sigma$. As $b<3/5$ by \eqref{eq:bBounds}, we further have
        \[
            f(2) \geq \frac{\pi - 3}{4} > 0,
        \]
        and thus conclude that $f(\sigma) > 0$ on $(1,2]$.

        Suppose finally that $\sigma > 2^{1/2}$. Then $\sigma^{-2} < 1/2< \tau_0$, by \Cref{lem:propertiesOfHomogeneousPhi}, so that
        \[
            \int_0^{\tau_0} \Phi(\tau)\min(1,\sigma \tau^{1/2})\dd \tau = \sigma \int_0^{\sigma^{-2}} \Phi(\tau)\tau^{1/2}\dd \tau + \int_{\sigma^{-2}}^{\tau_0} \Phi(\tau)\dd \tau,
        \]
        which we in turn can use in \eqref{eq:fDefinition} to compute that
        \[
            f'(\sigma) = \int_0^{\sigma^{-2}}\Phi(\tau)\tau^{1/2}\dd \tau - 2\sigma^{-3}b + \parn*{\frac{\pi}{2}+b}\sigma^{-2},
        \]
        for all $\sigma>2^{1/2}$. Here,
        \[
            \Phi(\tau)\tau^{1/2} = \parn*{\frac{1}{(1+\tau)^{1/2}}+\frac{1}{(1-\tau)^{1/2}}}\tau^{1/2} - 2 \geq 2\tau^{1/2}-2
        \]
        by \eqref{eq:homogeneousPhiDefinition} and the convexity of $H(\tau) = \abs{\tau}^{-1/2}$ on $\R^+$. Therefore, we infer that
        \begin{align*}
            f'(\sigma) & \geq 2\int_0^{\sigma^{-2}}(\tau^{1/2}-1)\dd \tau - 2\sigma^{-3}b + \parn*{\frac{\pi}{2}+b}\sigma^{-2} \\
                       & =\parn*{\frac{\pi}{2}+b-2}\sigma^{-2} + \parn*{\frac{4}{3}-2b}\sigma^{-3}                             \\
                       & >\frac{\pi-3}{2}\sigma^{-2} + \frac{2}{15}\sigma^{-3},
        \end{align*}
        for all $\sigma>2^{1/2}$, by using the bounds on $b$ provided by \eqref{eq:bBounds}. In conclusion, $f$ is increasing on $(2^{1/2},\infty)$, and therefore positive on $(1,\infty)$; seeing as it is positive on $(1,2]$.
    \end{proof}

    \subsection{The limit for the derivative}

        We shall now prove the final limit of \hyperref[thm:mainTheorem]{Main Theorem}. Whereas we in \Cref{ssec:derivativeLimitLogarithmic} first proved the limit, and then obtained uniform regularity via \Cref{lem:localModulus}, we will here prove a sharper form of Hölder regularity first. This regularity is then used to establish the limit. These two approaches are complementary. The counterpart to \eqref{eq:identityForLogarithmicDerivative} in this case is still useful, and becomes
        \[
            \delta_{2x}K(\tau x) = \delta_{2x}H(\tau x) + \delta_{2x}R(\tau x) = -\Gamma(\tau) + \delta_{2x}R(\tau x),
        \]
        where
        \begin{equation}
            \label{eq:gamma}
            \Gamma(\tau) \ceq \frac{1}{\abs{\tau-1}^{1/2}}-\frac{1}{(\tau+1)^{1/2}}
        \end{equation}
        is also positive on $\R^+$.

        To illustrate the idea, if one \emph{formally} differentiates the toy equation in \eqref{eq:homogeneousToyEquation}, then
        \begin{align*}
            2\parn*{\frac{u(x)}{x^{1/2}}}x^{1/2}u'(x) & = \pv{\int_0^\infty \delta_{2x}H'(y)u(y)\dd y}                                               \\
                                                      & = \pv{\int_0^\infty \Gamma'(\tau)}\tau^{1/2}\parn*{\frac{u(\tau x)}{(\tau x)^{1/2}}}\dd \tau
        \end{align*}
        for all $x > 0$. In principle, it should therefore be the case that
        \begin{equation}
            \label{eq:derivativeLimitPrincipalValue}
            \lim_{x \to 0}{x^{1/2}u'(x)} =  \frac{1}{2}\pv{\int_0^\infty\Gamma'(\tau)\tau^{1/2} \dd \tau}
        \end{equation}
        because of \Cref{prop:homogeneousAsymptoticBehavior}. This principal value integral can be shown to, in fact, equal $\pi/2$, so we find the ``correct'' limit. Of course, to rigorously justify this computation, especially for the full equation, we need to work harder.

        \begin{lemma}\label{lemma:Holder}
            There is some $\nu > 0$ so that
            \begin{equation}
                \label{eq:sharperHolderBound}
                x^{1/2}(u(x+h)-u(x-h)) \lesssim h
            \end{equation}
            for all $0 \leq h \leq x \leq \nu$. As a consequence, $u$ is $C^{1/2}$-Hölder continuous on $[0,\nu]$, and
            \[
                u'(x) \lesssim x^{-1/2}
            \]
            for $x \in (0,\nu]$.
        \end{lemma}

        \begin{proof}
            This proof is a variant of the proof of global regularity given for the Whitham equation in \cite{Ehrnstroem19Whithams}, but adapted to obtain more information than just Hölder continuity. The aim is to build up regularity by applying a a bootstrap argument to \eqref{eq:condensedFormulationEven}. We begin by noting that, if we introduce the notation $N(t) \ceq (1+n(t))t^2$ for the nonlinearity on the left-hand side of \eqref{eq:condensedFormulationEven}, and for simplicity extend $u$ to an even function on $\R$, then
            \begin{equation}
                \label{eq:doubleSymmetrisation}
                \begin{aligned}
                    N(u(x+h))-N(u(x-h)) & = \int_0^\infty (\delta_{x+h}^2 K(y) - \delta_{x-h}^2 K(y))u(y) \dd y \\
                                        & = -\int_0^\infty \delta_{2h}K(y)\delta_{2x}u(y)\dd y
                \end{aligned}
            \end{equation}
            for all $x,h \in \R$. This equation was referred to as a double symmetrisation formula in \cite{Ehrnstroem19Whithams}.

            On the left-hand side,
            \[
                N(u(x+h))-N(u(x-h)) = \int_{u(x-h)}^{u(x+h)} N'(t) = (1 + o(1))(u(x+h)^2-u(x-h)^2)
            \]
            for $0 < h < x$ as $x \to 0$. This is because $N'(t) = (2+o(1))t$ as $t \to 0$ by \Cref{ass:nonlinearity}, and because $u(0)=0$. Moreover, we further have
            \[
                u(x+h)^2-u(x-h)^2 = (u(x+h)+u(x-h))\delta_{2h}u(x) \eqsim x^{1/2} \delta_{2h}u(x)
            \]
            by \Cref{prop:homogeneousAsymptoticBehavior}.

            We next turn to the right-hand side of \eqref{eq:doubleSymmetrisation}, which we split as
            \begin{multline}
                \label{eq:doubleSymmetrisationRHS}
                -\int_0^\infty \delta_{2h}K(y)\delta_{2x}u(y)\dd y = -\parn*{\int_0^x + \int_x^\nu} \delta_{2h}H(y)\delta_{2x}u(y)\dd y\\
                -\int_0^\nu \delta_{2h}R(y)\delta_{2x}u(y)\dd y-\int_\nu^\infty \delta_{2h}K(y)\delta_{2x}u(y)\dd y
            \end{multline}
            for $0 < h < x \leq \nu$, with $H$ and $R$ as in \Cref{ass:integrabilitySymmetryAndConvexity}. Here, the final terms satisfy
            \begin{align*}
                \abs*{\int_\nu^\infty \delta_{2h}K(y)\delta_{2x}u(y)\dd y} & \leq 2h K(\nu-h) \norm{u}_{L^\infty}\lesssim h,           \\
                \abs*{\int_0^\nu \delta_{2h}R(y)\delta_{2x}u(y)\dd y}      & \leq 2\nu h\norm{R''}_{L^1}\norm{u}_{L^\infty}\lesssim h,
            \end{align*}
            where the first inequality follows by similar argument as in \Cref{lem:secondDifferenceProperties}, while the second follows from the bound on $R'$ from \Cref{lem:remainderMayBeDisregarded}. Furthermore, the first term on the right-hand side of \eqref{eq:doubleSymmetrisationRHS} satisfies
            \[
                \abs*{\int_0^x \delta_{2h} H(y)\delta_{2x}u(y)\dd y} = \frac{h}{x^{1/2}}\abs*{\int_0^{x/h}\Gamma(\tau)\delta_{2x}u(\tau h) \dd \tau} \lesssim h \int_0^\infty \Gamma(\tau)\dd \tau \lesssim h
            \]
            from \Cref{prop:homogeneousAsymptoticBehavior}, since $\Gamma$ from \eqref{eq:gamma} is integrable.

            In summary, we have demonstrated that

            \begin{equation}
                \label{eq:bootstrapWhitham}
                x^{1/2}\abs{\delta_{2h}u(x)} \lesssim  h^{1/2} \int_{x/h}^{\nu/h} \Gamma(\tau) \delta_{2x} u(\tau h)\dd \tau + h
            \end{equation}
            for $0 < h < x \leq \nu$, after possibly shrinking $\nu$. Define now the possibly infinite quantity
            \begin{equation}
                \label{eq:Cdef}
                C(\alpha) \ceq \sup_{0 < h < x \leq \nu}{x^\alpha\frac{\abs{\delta_{2h}u(x)}}{h^{2\alpha}}}
            \end{equation}
            for each $\alpha \in [0,1/2]$. As a function taking extended real values, $C$ is nondecreasing and left-continuous in $\alpha$, and is at least finite at $\alpha=0$. We emphasise that the supremum is \emph{not} taken over $\nu$.

            Let $\alpha$ be such that $C(\alpha)$ is finite. Then
            \[
                \abs{\delta_{2x}u(\tau h)} \lesssim \min(C(\alpha) x^{2\alpha} (\tau h)^{-\alpha},(\tau h)^{1/2}) \leq C(\alpha)^{1/2}x^\alpha (\tau h)^{1/2-\alpha/4}
            \]
            in the integrand in \eqref{eq:bootstrapWhitham}, by \Cref{prop:homogeneousAsymptoticBehavior} and the definition of $C(\alpha)$ in \eqref{eq:Cdef}. We have also used that $\min(a,b)\leq \sqrt{ab}$ for all $a,b \geq 0$. Inserting this into \eqref{eq:bootstrapWhitham}, we find
            \begin{equation}
                \label{eq:bootstrapWhithamInserted}
                \begin{aligned}
                    x^{1/2} \abs{\delta_{2h}u(x)} & \lesssim C(\alpha)^{1/2} h^{3/4-\alpha/2} x^\alpha \int_{x/h}^\infty \Gamma(\tau)\tau^{1/4-\alpha/2}\dd \tau + h \\
                                                  & \lesssim C(\alpha)^{1/2}hx^{\alpha/2-1/4}+h
                \end{aligned}
            \end{equation}
            for all $0 < h < x \leq \nu$. The second inequality comes from the fact that
            \[
                \int_z^\infty \Gamma(\tau)\tau^{1/4-\alpha/2}\dd \tau \lesssim \frac{1}{z^{1/4+\alpha/2}}
            \]
            uniformly in $z \geq 1$ and $\alpha$. If we now divide \eqref{eq:bootstrapWhithamInserted} by $h^{1/2 + \alpha}x^{1/4-\alpha/2}$, we arrive at
            \[
                x^{1/4+\alpha/2} \frac{\abs{\delta_{2h}u(x)}}{h^{1/2+\alpha}} \lesssim C(\alpha)^{1/2} \parn*{\frac{h}{x}}^{1/2-\alpha} + h^{1/4-\alpha/2} \parn*{\frac{h}{x}}^{1/4-\alpha/2}
            \]
            for $0 < h < x \leq \nu$. In particular, we thus have
            \begin{equation}
                \label{eq:bootstrapC}
                C(1/4+\alpha/2) \lesssim C(\alpha)^{1/2} + 1
            \end{equation}
            according to the definition in \eqref{eq:bootstrapWhithamInserted}. Crucially, the implicit constant does \emph{not} depend on $\alpha$.

            From \eqref{eq:bootstrapC}, we immediately conclude by induction that since $C(0)$ is finite, we have
            \[
                C\parn*{\frac{1}{2}-\frac{1}{2^k}} < \infty
            \]
            for all $k \geq 1$. Thus, since $C$ is nondecreasing, it is in fact finite for all $\alpha \in [0,1/2)$. Moreover,
            \[
                C(\alpha) \leq C(1/4+\alpha/2) \lesssim C(\alpha)^{1/2} + 1
            \]
            implies a uniform bound on $C(\alpha)$ for all $\alpha \in [0,1/2)$. Continuity from the left ensures that $C(1/2)$ is finite as well, concluding the proof.
        \end{proof}

        Using the regularity furnished by \eqref{eq:sharperHolderBound} in \Cref{lemma:Holder}, we are now able to fully justify a version of the calculation in \eqref{eq:derivativeLimitPrincipalValue}. We purposely avoid having to deal with principal value integrals.

        \begin{proposition}
            \label{prop:homogeneousderivative}
            Under \Cref{ass:integrabilitySymmetryAndConvexity,ass:nonlinearity,ass:solutionU}, the derivative of the solution enjoys the limit
            \[
                \lim_{x \to 0}{\frac{u'(x)}{x^{-1/2}}} = \frac{\pi}{4}
            \]
            when $K$ has a homogeneous singularity.
        \end{proposition}

        \begin{proof}
            We return to \eqref{eq:doubleSymmetrisation}, make the same splitting of the right-hand side as in \eqref{eq:doubleSymmetrisationRHS}, and divide by $2h$; so as to get
            \begin{multline} \label{eq:doubleSymmetrisationRestatedForSimplicity}
                \frac{N(u(x+h))-N(u(x-h))}{2h} =  -\frac{1}{2h}\parn*{\int_0^x + \int_x^\nu} \delta_{2h}H(y)\delta_{2x}u(y)\dd y\\
                -\frac{1}{2h}\int_0^\nu \delta_{2h}R(y)\delta_{2x}u(y)\dd y-\frac{1}{2h}\int_\nu^\infty \delta_{2h}K(y)\delta_{2x}u(y)\dd y
            \end{multline}
            for all $0<x\leq \nu$ and $0<h\leq x/2$. The intention is to obtain the desired limit from this equation, by first letting $h \to 0$, and subsequently $x \to 0$: After doing so to the left-hand side of \eqref{eq:doubleSymmetrisationRestatedForSimplicity}, it reads
            \begin{equation}\label{eq: limitOfTheLHSInTheLimitOfTheDerivativeForTheHomogeneousCase}
                \adjustlimits\lim_{x\to0}\lim_{h\to 0}{\frac{N(u(x+h))-N(u(x-h))}{2h}} =  \lim_{x\to0}{ N'(u(x))u'(x)} = \pi\lim_{x\to0}{\frac{u'(x)}{x^{-1/2}}},
            \end{equation}
            where the final equality follows from $N'(u(x))=(2+o(1))u(x)$ and \Cref{prop:homogeneousAsymptoticBehavior}. Of course, we do not yet know that this limit actually exists.

            Turning our attention to the right-hand side of \eqref{eq:doubleSymmetrisationRestatedForSimplicity}, we next establish the limit of each integral separately. To accomplish this, we shall prove that the integrands are dominated by integrable functions, independently of $x$ and $h$. Consequently, limits and integrals may be interchanged, by using the dominated convergence theorem.

            For the first integral, we use the symmetry $\delta_{2x}u(y)=\delta_{2y}u(x)$ to write
            \begin{equation}\label{eq:splittingTheFirstIntegralOnTheRHSUp}
                \begin{aligned}
                    -\frac{1}{2h}\int_0^x \delta_{2h}H(y)\delta_{2x}u(y)\dd y & =-\frac{1}{2h}\parn*{\int_0^{2h} + \int_{2h}^x} \delta_{2h}H(y)\delta_{2y}u(x)\dd y                                                                \\
                    & =\begin{multlined}[t]\frac{1}{2h^{1/2}}\int_0^2 \Gamma(\tau) \delta_{2\tau h} u(x)\dd\tau\\- \int_{2h/x}^1 \frac{\delta_{2h}H(\tau x)}{2h}\delta_{2\tau x}u(x)x\dd \tau,\end{multlined}
                \end{aligned}
            \end{equation}
            for all $0< 2h\leq x \leq \nu$. In the first term on the right-hand side,
            \begin{equation}
                \label{eq:derivativeLimitFirstIntegralFirstTerm}
                \frac{1}{2h^{1/2}} \Gamma(\tau) \delta_{2\tau h}u(x) = \frac{1}{2}\parn*{\frac{h}{x}}^{1/2} \Gamma(\tau)\tau \parn*{\frac{x^{1/2}\delta_{2\tau h}u(x)}{\tau h}},
            \end{equation}
            and therefore
            \[
                \abs*{\frac{1}{2h^{1/2}} \Gamma(\tau) \delta_{2\tau h}u(x)}\lesssim \Gamma(\tau)\tau
            \]
            by \Cref{lemma:Holder}. Because $\tau\mapsto \Gamma(\tau)\tau$ is integrable on $[0,2]$, it follows from \eqref{eq:derivativeLimitFirstIntegralFirstTerm} and dominated convergence that the first integral on the right hand-side of \eqref{eq:splittingTheFirstIntegralOnTheRHSUp} vanishes as $h\to 0$, for each $0<x\leq \nu$.

            As for the second term on the right-hand side of \eqref{eq:splittingTheFirstIntegralOnTheRHSUp}, its integrand may be expressed as
            \[
                - \frac{\delta_{2h}H(\tau x)}{2h}\delta_{2\tau x}u(x)x= \frac{\tau}{\parn*{\tau^2-(h/x)^2}^{1/2}\parn*{(\tau+h/x)^{1/2}+(\tau-h/x)^{1/2}}} \parn*{\frac{x^{1/2}\delta_{2\tau x} u(x)}{\tau x}},
            \]
            and is thus dominated by
            \begin{equation}
                \label{eq:derivativeLimitFirstIntegralSecondTermBound}
                \abs*{- \frac{\delta_{2h}H(\tau x)}{2h}\delta_{2\tau x}u(x)x} \lesssim \frac{1}{\tau^{1/2}}
            \end{equation}
            for all $0 < 2h\leq  x \leq \nu$ and $2h/x \leq \tau \leq 1$. Moreover, we have the limit
            \[
                \lim_{h \to 0}{\parn*{- \frac{\delta_{2h}H(\tau x)}{2h}\delta_{2\tau x}u(x)x} = \frac{1}{2\tau^{3/2}}} \frac{\delta_{2\tau x}u(x)}{x^{1/2}}
            \]
            for each fixed $x \in (0,\nu]$ and $\tau \in (0,1]$, where in turn
            \[
                \lim_{x \to 0}{\parn*{\frac{1}{2\tau^{3/2}} \frac{\delta_{2\tau x}u(x)}{x^{1/2}}}}=\frac{\pi}{4\tau^{3/2}}\parn*{(1+\tau)^{1/2}-(1-\tau)^{1/2}}
            \]
            for each $\tau \in (0,1]$ by \Cref{prop:homogeneousAsymptoticBehavior}. Since the upper bound \eqref{eq:derivativeLimitFirstIntegralSecondTermBound} is integrable on $[0,1]$, we may therefore conclude that
            \begin{equation}\label{eq: limitOfTheFirstIntegralOnTheRHSForLimitOfDerivativeInHomogeneousCase}
                \adjustlimits\lim_{x\to 0}\lim_{h\to0}{\parn*{-\frac{1}{2h}\int_0^x \delta_{2h}H(y)\delta_{2x}u(y)\dd y}} =\int_0^1\frac{\pi}{4\tau^{3/2}}\parn*{(1+\tau)^{1/2}-(1-\tau)^{1/2}}\dd \tau
            \end{equation}
            from \eqref{eq:splittingTheFirstIntegralOnTheRHSUp} and the pointwise limits.

            We move on to the second integral on the right-hand side of \eqref{eq:doubleSymmetrisationRestatedForSimplicity}. It can be written as
            \[
                -\frac{1}{2h}\int_x^\nu \delta_{2h} H(y) \delta_{2x}u(y)\dd y = - \int_{1}^{\nu/x} \frac{\delta_{2h}H(\tau x)}{2h}\delta_{2x}u(\tau x)x\dd \tau,
            \]
            where \Cref{lemma:Holder} implies that the integrand
            \[
                -\frac{\delta_{2h}H(\tau x)}{2h}\delta_{2x}u(\tau x)x = \frac{\tau^{-1/2}}{\parn*{\tau^2-(h/x)^2}^{1/2}\parn*{(\tau+h/x)^{1/2}+(\tau-h/x)^{1/2}}} \parn*{\frac{(\tau x)^{1/2}\delta_{2x} u(\tau x)}{x}}
            \]
            is bounded by
            \begin{equation}
                \label{eq:derivativeLimitSecondIntegralBound}
                \abs*{-\frac{\delta_{2h}H(\tau x)}{2h}\delta_{2x}u(\tau x)x} \lesssim \frac{1}{\tau^2}
            \end{equation}
            for all $0 < 2h\leq x \leq \nu$ and $2h/x \leq \tau \leq \nu/x$. Furthermore, it admits the limit
            \[
                \lim_{h \to 0}{\parn*{-\frac{\delta_{2h}H(\tau x)}{2h}\delta_{2x}u(\tau x)x}} = \frac{1}{2\tau^{3/2}} \frac{\delta_{2x}u(\tau x)}{x^{1/2}}
            \]
            for each fixed $x \in (0,\nu]$ and $\tau \in (0,\nu/x]$, and we further have
                \[
                \lim_{x \to 0}{\parn*{\frac{1}{2\tau^{3/2}} \frac{\delta_{2x}u(\tau x)}{x^{1/2}}}} = \frac{\pi}{4\tau^{3/2}}\parn*{(1+\tau)^{1/2}-(\tau-1)^{1/2}}
                \]
                for every $\tau > 0$ by \Cref{prop:homogeneousAsymptoticBehavior}. The upper bound in \eqref{eq:derivativeLimitSecondIntegralBound} is integrable on $[1,\infty)$, and we therefore have
            \begin{equation}\label{eq: limitOfTheSecondIntegralOnTheRHSForLimitOfDerivativeInHomogeneousCase}
                \adjustlimits\lim_{x\to 0}\lim_{h\to0}{\parn*{-\frac{1}{2h}\int_{x}^\nu \delta_{2h}H(y)\delta_{2x}u(y)\dd y }}=\int_1^\infty\frac{\pi}{4\tau^{3/2}}\parn*{(1+\tau)^{1/2}-(\tau-1)^{1/2}}\dd \tau.
            \end{equation}

            For the third integral on the right-hand side of \eqref{eq:doubleSymmetrisationRestatedForSimplicity}, we use the bound on $R'$ from  \Cref{lem:remainderMayBeDisregarded} to find that the integrand is dominated by
            \[
                \abs*{\frac{\delta_{2h}R(y)}{2h} \delta_{2x}u(y)} \lesssim 1
            \]
            for all $0 < 2h\leq x \leq \nu$ and $0 < y < \nu$. We note also the two limits
            \[
                \lim_{h \to 0}{\parn*{\frac{\delta_{2h}R(y)}{2h} \delta_{2x}u(y)}} = R'(y)\delta_{2x}u(y)
            \]
            for every $0 < x,y \leq \nu$, and
            \[
                \lim_{x \to 0}{\parn*{R'(y)\delta_{2x}u(y)}} = 0
            \]
            for all $0 < y \leq \nu$. As before, we conclude that
            \begin{equation}\label{eq: limitOfTheThirdIntegralOnTheRHSForLimitOfDerivativeInHomogeneousCase}
                \adjustlimits\lim_{x\to 0}\lim_{h\to0}{\parn*{-\frac{1}{2h}\int_0^\nu \delta_{2h}R(y)\delta_{2x}u(y)\dd y}} =0
            \end{equation}
            by using dominated convergence.

            Finally, for the fourth integral on the right-hand side of \eqref{eq:doubleSymmetrisationRestatedForSimplicity} we exploit the convexity of $K$ to see that the integrand is dominated by
            \begin{equation}
                \label{eq:derivativeLimitFourthIntegralBound}
                \abs*{\frac{\delta_{2h}K(y)}{2h} \delta_{2x}u(y)} \lesssim -K'(y-\nu/2)
            \end{equation}
            for all $0 < x \leq \nu$, $0 < h < x/2$, and $y \geq \nu$.  We observe also that
            \[
                \lim_{h \to 0}{\parn*{\frac{\delta_{2h}K(y)}{2h} \delta_{2x}u(y)}} = K'(y)\delta_{2x}u(y)
            \]
            for every $0 < x \leq \nu \leq y$, and that
            \[
                \lim_{x \to 0}{\parn*{K'(y)\delta_{2x}u(y)}}=0
            \]
            for all $y \geq \nu$, so this integral also vanishes in the limit:
            \begin{equation}\label{eq: limitOfTheFourthIntegralOnTheRHSForLimitOfDerivativeInHomogeneousCase}
                \lim_{x\to 0}\lim_{h\to0} -\frac{1}{2h}\int_\nu^\infty \delta_{2h}K(y)\delta_{2x}u(y)\dd y =0.
            \end{equation}

            In summary, the four limits \eqref{eq: limitOfTheFirstIntegralOnTheRHSForLimitOfDerivativeInHomogeneousCase}, \eqref{eq: limitOfTheSecondIntegralOnTheRHSForLimitOfDerivativeInHomogeneousCase}, \eqref{eq: limitOfTheThirdIntegralOnTheRHSForLimitOfDerivativeInHomogeneousCase}, and \eqref{eq: limitOfTheFourthIntegralOnTheRHSForLimitOfDerivativeInHomogeneousCase} show that we can conclude from equation \eqref{eq:doubleSymmetrisationRestatedForSimplicity} that \eqref{eq: limitOfTheLHSInTheLimitOfTheDerivativeForTheHomogeneousCase} exists, and therefore that
            \[
                \lim_{x \to 0}{\frac{u'(x)}{x^{-1/2}}} = \frac{1}{4}\int_0^\infty \frac{(1+\tau)^{1/2}-\abs{1-\tau}^{1/2}}{\tau^{3/2}}\dd \tau= \frac{\pi}{4},
            \]
            after dividing by $\pi$. The last equality holds by an argument similar to the one utilised in the proof of \Cref{lem:toyEquation}.
        \end{proof}

        We remark that the method of proof in \Cref{lemma:Holder,prop:homogeneousderivative} may be repeated inductively: It can be seen from the proof of \Cref{prop:homogeneousderivative} that $u'$ satisfies
        \begin{multline*}
            N'(u(x))u'(x) = \int_0^1 \frac{1}{2\tau^{3/2}}\frac{\delta_{2\tau x}u(x)}{x^{1/2}}\dd \tau + \int_1^\infty \frac{1}{2\tau^{3/2}}\frac{\delta_{2x}u(\tau x)}{x^{1/2}}\dd \tau\\
            -\int_0^\nu R'(y)\delta_{2x}u(y)\dd y - \int_\nu^\infty K'(y)\delta_{2x}u(y)\dd y
        \end{multline*}
        for all $0 < x \leq \nu$ on which similar analysis can be applied to study $u''$. More generally, given a nonlinearity $n$ that is $C^N$ and vanishing up to order $N-1$ at the origin, a more regular $R$, a solution $u$ that is $C^N(0,\nu]$, with estimates $u^{(k)}(x) =  \parn*{\pi/2 + o(1) } \mr{D}^k x^{1/2}$ and higher-order analogues of \Cref{lemma:Holder} for $k = 0, 1,2, \ldots N-1$, one may prove that $u^{N}(x) =  \parn*{ \pi/2 + o(1) } \mr{D}^{N} x^{1/2}$. This is done by taking another difference in the equation satisfied by $u^{(N-1)}$, establishing its analogue of \Cref{lemma:Holder}, and then progressing in a similar manner to the proof of \Cref{prop:homogeneousderivative}. We refrain from pursuing this, and will be content with asymptotics for the first derivative.

    \subsection{The Whitham equation}
        \label{sec:applicationForHighestWavesOfUnidirectionalWhithamEquations}
        Proceeding as for its bidirectional counterpart in \Cref{sec:applicationForTheHighestWavesOfBidirectionalWhithamEquation}, we insert the steady-wave ansatz $\phi(x,t)=\varphi(x-ct)$ into \eqref{eq:unidirectionalWhitham} to arrive at the \textit{steady Whitham equation}
        \begin{equation}
            \label{eq:steadyUnidirectionalWhitham}
            K_W * \varphi=\varphi(c-\varphi)
        \end{equation}
        after integration. The integration constant is chosen to be zero, which can be done by Galilean transformation, like in \cite{Ehrnstroem19Whithams}. Again, one observes that the right-hand side of \eqref{eq:steadyUnidirectionalWhitham} is increasing to the left of its maximum at $\varphi =c/2$. This is the height of a highest wave for this equation. If $\varphi$ assumes this height at the origin, and is even, then
        \[
            u \ceq \sqrt{2\pi}\parn*{\frac{c}{2}-\varphi}
        \]
        satisfies the equation
        \[
            u(x)^2 = \int_0^\infty \delta_x^2(\sqrt{2\pi}K_W)(y)u(y)\dd y,
        \]
        which is of the desired form \eqref{eq:condensedFormulationEven}. It is immediate that \Cref{ass:nonlinearity} holds, and it is well-known \cite{Ehrnstroem19Whithams} that \Cref{ass:integrabilitySymmetryAndConvexity} is satisfied. See also the comment on the precise behaviour of $K_W$ in \Cref{rem:the series expansion} below.

        The existence of a limiting $2\pi$-periodic solution $(\varphi,c)$ of \eqref{eq:steadyUnidirectionalWhitham} was proved in \cite{Ehrnstroem19Whithams}. This solution is even, assumes $\varphi(0) = c/2$, decreases on $(0,\pi)$, and is smooth on $(0,2\pi)$. In particular, \Cref{ass:solutionU} holds. All assumptions required for \Cref{prop:homogeneousAsymptoticBehavior} and \Cref{prop:homogeneousderivative} are therefore satisfied, and we may settle a conjecture posed in the aforementioned paper. Our result also applies equally well to the highest \emph{solitary} waves recently found in \cite{Truong22Global} and \cite{Ehrnstroem23Direct}. As shown in \cite{Bruell17Symmetry,Bruell23Symmetry}, such solitary waves are necessarily even, and smooth and decreasing on $\R^+$.

        \begin{corollary}
            Let $\varphi$ be a solution of the steady Whitham equation \eqref{eq:steadyUnidirectionalWhitham} that is even, achieves $\varphi(0)=c/2$, and is smooth and decreasing on a nonempty interval $(0,\nu)$. Then
            \begin{align*}
                \varphi(x)  & = \frac{c}{2} - \parn*{\sqrt{\frac{\pi}{8}}+o(1)}x^{1/2} \\
                \varphi'(x) & = - \frac{1}{2}\parn*{\sqrt{\frac{\pi}{8}}+o(1)}x^{-1/2}
            \end{align*}
            as $x \searrow 0$.
        \end{corollary}

        \begin{remark}
            \label{rem:the series expansion}
            It is possible to give the series expansion
            \[
                K_W(x) = \frac{1}{\sqrt{2\pi}}\sum_{n=0}^\infty (-1)^n \binom{\floor{n/2}-1/2}{\floor{n/2}}\parn*{\frac{2n + \sqrt{4n^2+x^2}}{4n^2+x^2}}^{1/2}.
            \]
            for the Whitham kernel. To the best of our knowledge, this expansion of the kernel for standard linear gravity wave dispersion is new. We note, in particular, that the first term is precisely the singular part of the kernel. As written, the series is only conditionally convergent, but this can be remedied by merging the terms corresponding to $n=2k-1$ and $n=2k$ for $k \geq 1$. These all become smooth, even, negative, and increasing on $\R^+$.

            To prove the series expansion, the key observation is that
            \[
                \parn*{\frac{\tanh(\xi)}{\xi}}^{1/2} = \sum_{n=0}^\infty \binom{\floor{n/2}-1/2}{\floor{n/2}}(-1)^n\frac{e^{-2\abs{\xi}n}}{\abs{\xi}^{1/2}}
            \]
            for $\xi \neq 0$, which can be seen by writing the numerator in terms of a binomial series. Here
            \begin{align*}
                \frac{1}{2\pi}\int_\R \frac{e^{-2\abs{\xi}n}}{\abs{\xi}^{1/2}}e^{i\xi x}\dd \xi & = \frac{1}{\pi}\re{\int_0^\infty \frac{e^{-(2n-ix)\xi}}{\xi^{1/2}}\dd\xi}=\frac{1}{\sqrt{\pi}}\re{(2n-ix)^{-1/2}} \\
                                                                                                & =\frac{1}{\sqrt{2\pi}}\parn*{\frac{2n + \sqrt{4n^2+x^2}}{4n^2+x^2}}^{1/2}
            \end{align*}
            for all $n \geq 1$, and in the sense of distributions when $n=0$.
        \end{remark}
    \printbibliography[title=References]

@Article{AcevesSanchez13Numerical,
  author   = {Aceves-Sánchez, P. and Minzoni, A. A. and Panayotaros, P.},
  journal  = {Wave Motion},
  title    = {Numerical study of a nonlocal model for water-waves with variable depth},
  year     = {2013},
  issn     = {0165-2125},
  number   = {1},
  pages    = {80--93},
  volume   = {50},
  doi      = {10.1016/j.wavemoti.2012.07.002},
  fjournal = {Wave Motion. An International Journal Reporting Research on Wave Phenomena},
}

@Article{Afram21Steady,
  author   = {Afram, Obed Opoku},
  journal  = {Trans. R. Norw. Soc. Sci. Lett.},
  title    = {On steady solutions of a generalized {W}hitham equation},
  year     = {2021},
  issn     = {1893-9708},
  pages    = {5--29},
  volume   = {3},
  doi      = {11250/3054570},
  fjournal = {Transactions of the Royal Norwegian Society of Science and Letters},
}

@Article{Amick82Stokes,
  author   = {Amick, C. J. and Fraenkel, L. E. and Toland, J. F.},
  journal  = {Acta Math.},
  title    = {On the {S}tokes conjecture for the wave of extreme form},
  year     = {1982},
  issn     = {0001-5962},
  pages    = {193--214},
  volume   = {148},
  doi      = {10.1007/BF02392728},
  fjournal = {Acta Mathematica},
}

@Article{Arnesen19Non,
  author  = {Arnesen, Mathias Nikolai},
  journal = {J. Math. Anal. Appl.},
  title   = {A non-local approach to waves of maximal height for the {Degasperis}--{Procesi} equation},
  year    = {2019},
  issn    = {0022-247X},
  number  = {1},
  pages   = {25--44},
  volume  = {479},
  doi     = {10.1016/j.jmaa.2019.06.014},
}

@Article{Arnesen22Decay,
  author  = {Arnesen, Mathias Nikolai},
  journal = {J. Math. Anal. Appl.},
  title   = {Decay and symmetry of solitary waves},
  year    = {2022},
  issn    = {0022-247X},
  number  = {1},
  pages   = {125450},
  volume  = {507},
  doi     = {10.1016/j.jmaa.2021.125450},
}

@Article{Bruell17Symmetry,
  author  = {Bruell, Gabriele and Ehrnström, Mats and Pei, Long},
  journal = {J. Differential Equations},
  title   = {Symmetry and decay of traveling wave solutions to the {Whitham} equation},
  year    = {2017},
  issn    = {0022-0396},
  number  = {8},
  pages   = {4232--4254},
  volume  = {262},
  doi     = {10.1016/j.jde.2017.01.011},
}

@Article{Bruell21Waves,
  author  = {Bruell, Gabriele and Dhara, Raj Narayan},
  journal = {Indiana Univ. Math. J.},
  title   = {Waves of maximal height for a class of nonlocal equations with homogeneous symbols},
  year    = {2021},
  issn    = {0022-2518},
  number  = {2},
  pages   = {711--742},
  volume  = {70},
  doi     = {10.1512/iumj.2021.70.8368},
}

@Article{Bruell23Symmetry,
  author   = {Bruell, Gabriele and Pei, Long},
  journal  = {SIAM J. Math. Anal.},
  title    = {Symmetry of periodic traveling waves for nonlocal dispersive equations},
  year     = {2023},
  issn     = {0036-1410},
  number   = {1},
  pages    = {486--507},
  volume   = {55},
  doi      = {10.1137/21M1433162},
  fjournal = {SIAM Journal on Mathematical Analysis},
}

@Article{Carter18Bidirectional,
  author  = {Carter, John D.},
  journal = {Wave Motion},
  title   = {Bidirectional {Whitham} equations as models of waves on shallow water},
  year    = {2018},
  issn    = {0165-2125},
  pages   = {51--61},
  volume  = {82},
  doi     = {10.1016/j.wavemoti.2018.07.004},
}

@Article{Dahne23Highest,
  author   = {Dahne, Joel and Gómez-Serrano, Javier},
  journal  = {Arch. Ration. Mech. Anal.},
  title    = {Highest cusped waves for the {B}urgers-{H}ilbert equation},
  year     = {2023},
  issn     = {0003-9527},
  number   = {5},
  pages    = {74},
  volume   = {247},
  doi      = {10.1007/s00205-023-01904-6},
  fjournal = {Archive for Rational Mechanics and Analysis},
}

@Article{Edmunds00Embeddings,
  author  = {Edmunds, David E. and Haroske, Dorothee D.},
  journal = {J. Approx. Theory},
  title   = {Embeddings in spaces of {Lipschitz} type, entropy and approximation numbers, and applications},
  year    = {2000},
  issn    = {0021-9045},
  number  = {2},
  pages   = {226--271},
  volume  = {104},
  doi     = {10.1006/jath.2000.3453},
}

@InProceedings{Ehrnstroem15Whithams,
  author    = {Ehrnström, M.},
  booktitle = {Oberwolfach {R}eport {N}o. 19},
  title     = {On {Whitham}'s conjecture of a highest cusped wave for a nonlocal shallow water wave equation},
  year      = {2015},
}

@Article{Ehrnstroem19Existence,
  author  = {Ehrnström, Mats and Johnson, Mathew A. and Claassen, Kyle M.},
  journal = {Arch. Ration. Mech. Anal.},
  title   = {Existence of a highest wave in a fully dispersive two-way shallow water model},
  year    = {2019},
  issn    = {0003-9527},
  number  = {3},
  pages   = {1635--1673},
  volume  = {231},
  doi     = {10.1007/s00205-018-1306-5},
}

@Article{Ehrnstroem19Whithams,
  author  = {Ehrnström, Mats and Wahlén, Erik},
  journal = {Ann. Inst. H. Poincaré Anal. Non Linéaire},
  title   = {On {Whitham}'s conjecture of a highest cusped wave for a nonlocal dispersive equation},
  year    = {2019},
  issn    = {0294-1449},
  number  = {6},
  pages   = {1603--1637},
  volume  = {36},
  doi     = {10.1016/j.anihpc.2019.02.006},
}

@Article{Ehrnstroem23Direct,
  author   = {Ehrnström, Mats and Nik, Katerina and Walker, Christoph},
  journal  = {Proc. Amer. Math. Soc.},
  title    = {A direct construction of a full family of {W}hitham solitary waves},
  year     = {2023},
  issn     = {0002-9939},
  number   = {3},
  pages    = {1247--1261},
  volume   = {151},
  doi      = {10.1090/proc/16191},
  fjournal = {Proceedings of the American Mathematical Society},
}

@Article{Emerald21Rigorous,
  author  = {Emerald, Louis},
  journal = {SIAM J. Math. Anal.},
  title   = {Rigorous derivation from the water waves equations of some full dispersion shallow water models},
  year    = {2021},
  issn    = {0036-1410},
  number  = {4},
  pages   = {3772--3800},
  volume  = {53},
  doi     = {10.1137/20M1332049},
}

@Article{Emerald21Rigorousa,
  author  = {Emerald, Louis},
  journal = {Nonlinearity},
  title   = {Rigorous derivation of the {Whitham} equations from the water waves equations in the shallow water regime},
  year    = {2021},
  issn    = {0951-7715},
  number  = {11},
  pages   = {7470--7509},
  volume  = {34},
  doi     = {10.1088/1361-6544/ac24df},
}

@Article{Enciso18Convexity,
  author        = {Enciso, Alberto and Gómez-Serrano, Javier and Vergara, Bruno},
  title         = {Convexity of Whitham's highest cusped wave},
  year          = {2018},
  archiveprefix = {arXiv},
  eprint        = {1810.10935},
  primaryclass  = {math.AP},
}

@Article{Fjordholm18Sharp,
  author  = {Fjordholm, Ulrik Skre},
  journal = {C. R. Math. Acad. Sci. Paris},
  title   = {Sharp uniqueness conditions for one-dimensional, autonomous ordinary differential equations},
  year    = {2018},
  issn    = {1631-073X},
  number  = {9},
  pages   = {916--921},
  volume  = {356},
  doi     = {10.1016/j.crma.2018.07.008},
}

@Article{Geyer19Linear,
  author  = {Geyer, Anna and Pelinovsky, D.},
  journal = {SIAM J. Math. Anal.},
  title   = {Linear instability and uniqueness of the peaked periodic wave in the reduced {Ostrovsky} equation},
  year    = {2019},
  issn    = {0036-1410},
  number  = {2},
  pages   = {1188--1208},
  volume  = {51},
  doi     = {10.1137/18M117978X},
}

@Book{Grafakos14Modern,
  author    = {Grafakos, Loukas},
  publisher = {Springer, New York},
  title     = {Modern {Fourier} analysis},
  year      = {2014},
  isbn      = {978-1-4939-1229-2},
  series    = {Graduate Texts in Mathematics},
  volume    = {250},
  doi       = {10.1007/978-1-4939-1230-8},
  pages     = {xvi+624},
}

@Article{Hildrum23Periodic,
  author   = {Hildrum, Fredrik and Xue, Jun},
  journal  = {J. Differential Equations},
  title    = {Periodic {H}ölder waves in a class of negative-order dispersive equations},
  year     = {2023},
  issn     = {0022-0396},
  pages    = {752--789},
  volume   = {343},
  doi      = {10.1016/j.jde.2022.10.023},
  fjournal = {Journal of Differential Equations},
  mrnumber = {4504564},
}

@Article{Hur17Wave,
  author   = {Hur, Vera Mikyoung},
  journal  = {Adv. Math.},
  title    = {Wave breaking in the {W}hitham equation},
  year     = {2017},
  issn     = {0001-8708},
  pages    = {410--437},
  volume   = {317},
  doi      = {10.1016/j.aim.2017.07.006},
  fjournal = {Advances in Mathematics},
}

@Article{Klein18Whitham,
  author   = {Klein, Christian and Linares, Felipe and Pilod, Didier and Saut, Jean-Claude},
  journal  = {Stud. Appl. Math.},
  title    = {On {W}hitham and related equations},
  year     = {2018},
  issn     = {0022-2526},
  number   = {2},
  pages    = {133--177},
  volume   = {140},
  doi      = {10.1111/sapm.12194},
  fjournal = {Studies in Applied Mathematics},
}

@Book{Lannes13Water,
  author    = {Lannes, David},
  publisher = {American Mathematical Society, Providence, RI},
  title     = {The Water Waves Problem},
  year      = {2013},
  isbn      = {978-0-8218-9470-5},
  series    = {Mathematical Surveys and Monographs},
  volume    = {188},
  doi       = {10.1090/surv/188},
  pages     = {xx+321},
}

@Article{Le22Waves,
  author  = {Le, Hung},
  journal = {Asymptot. Anal.},
  title   = {Waves of maximal height for a class of nonlocal equations with inhomogeneous symbols},
  year    = {2022},
  issn    = {0921-7134},
  number  = {4},
  pages   = {355--380},
  volume  = {127},
  doi     = {10.3233/asy-211694},
}

@Article{Moldabayev15Whitham,
  author  = {Moldabayev, Daulet and Kalisch, Henrik and Dutykh, Denys},
  journal = {Phys. D},
  title   = {The {Whitham} equation as a model for surface water waves},
  year    = {2015},
  issn    = {0167-2789},
  pages   = {99--107},
  volume  = {309},
  doi     = {10.1016/j.physd.2015.07.010},
}

@Article{Nilsson19Solitary,
  author  = {Nilsson, Dag and Wang, Yuexun},
  journal = {Z. Angew. Math. Phys.},
  title   = {Solitary wave solutions to a class of {Whitham}--{Boussinesq} systems},
  year    = {2019},
  issn    = {0044-2275},
  number  = {3},
  pages   = {70},
  volume  = {70},
  doi     = {10.1007/s00033-019-1116-0},
}

@Book{Oberhettinger90Tables,
  author    = {Oberhettinger, F.},
  publisher = {Springer-Verlag, Berlin},
  title     = {Tables of {Fourier} transforms and {Fourier} transforms of distributions},
  year      = {1990},
  isbn      = {3-540-50630-6},
  doi       = {10.1007/978-3-642-74349-8},
  pages     = {viii+259},
}

@Article{Oerke22Highest,
  author        = {Ørke, Magnus C.},
  title         = {Highest waves for fractional Korteweg--De Vries and Degasperis--Procesi equations},
  year          = {2022},
  archiveprefix = {arXiv},
  eprint        = {2201.13159},
  primaryclass  = {math.AP},
}

@Article{Pei19Note,
  author   = {Pei, Long and Wang, Yuexun},
  journal  = {Appl. Math. Lett.},
  title    = {A note on well-posedness of bidirectional {W}hitham equation},
  year     = {2019},
  issn     = {0893-9659},
  pages    = {215--223},
  volume   = {98},
  doi      = {10.1016/j.aml.2019.06.015},
  fjournal = {Applied Mathematics Letters. An International Journal of Rapid Publication},
  mrnumber = {3975130},
}

@Article{Plotnikov02Proof,
  author  = {Plotnikov, P. I.},
  journal = {Stud. Appl. Math.},
  title   = {A proof of the {Stokes} conjecture in the theory of surface waves},
  year    = {2002},
  issn    = {0022-2526},
  note    = {Translated from Dinamika Sploshn. Sredy No. 57 (1982)},
  number  = {2},
  pages   = {217--244},
  volume  = {108},
  doi     = {10.1111/1467-9590.01408},
}

@Article{Saut22Wave,
  author   = {Saut, Jean-Claude and Wang, Yuexun},
  journal  = {SIAM J. Math. Anal.},
  title    = {The wave breaking for {W}hitham-type equations revisited},
  year     = {2022},
  issn     = {0036-1410},
  number   = {2},
  pages    = {2295--2319},
  volume   = {54},
  doi      = {10.1137/20M1345207},
  fjournal = {SIAM Journal on Mathematical Analysis},
}

@Book{Taylor11Partial,
  author    = {Taylor, Michael E.},
  publisher = {Springer, New York},
  title     = {Partial differential equations {III}. {Nonlinear} equations},
  year      = {2011},
  edition   = {Second},
  isbn      = {978-1-4419-7048-0},
  series    = {Applied Mathematical Sciences},
  volume    = {117},
  doi       = {10.1007/978-1-4419-7049-7},
  pages     = {xxii+715},
}

@Article{Truong22Global,
  author   = {Truong, Tien and Wahlén, Erik and Wheeler, Miles H.},
  journal  = {Math. Ann.},
  title    = {Global bifurcation of solitary waves for the {W}hitham equation},
  year     = {2022},
  issn     = {0025-5831},
  number   = {3-4},
  pages    = {1521--1565},
  volume   = {383},
  doi      = {10.1007/s00208-021-02243-1},
  fjournal = {Mathematische Annalen},
}

@Article{Whitham67Variational,
  author    = {Whitham, G. B.},
  journal   = {Proc. A.},
  title     = {Variational methods and applications to water waves},
  year      = {1967},
  issn      = {0080-4630},
  pages     = {6--25},
  volume    = {299},
  publisher = {Royal Society of London, London},
  zbl       = {0163.21104},
}

@Book{Whitham74Linear,
  author    = {Whitham, G. B.},
  publisher = {Wiley-Interscience, New York-London-Sydney},
  title     = {Linear and nonlinear waves},
  year      = {1974},
  series    = {Pure and Applied Mathematics},
  pages     = {xvi+636},
}
    \end{document}